\documentclass[3p]{elsarticle}
\usepackage{hyperref}
\biboptions{authoryear}

\usepackage{mathrsfs}
\usepackage{amsfonts, amssymb, amsmath, amsthm, esint}
\usepackage{lipsum}
\usepackage{multirow}
\usepackage{latexsym}
\usepackage{latexsym}
\usepackage{tikz}
\usepackage{pgfplots}
\usepackage{natbib}
\usepackage{subcaption}
\usepackage{caption}

\newtheorem{proposition}{Proposition}
\newdefinition{definition}{Definition}
\begin{document}

\begin{frontmatter}

\title{A branch-and-price algorithm for the robust single-source capacitated facility location problem under demand uncertainty}
%\tnotetext[mytitlenote]{Fully documented templates are available in the elsarticle package on \href{http://www.ctan.org/tex-archive/macros/latex/contrib/elsarticle}{CTAN}.}

%% Group authors per affiliation:
\author{Jaehyeon Ryu\fnref{myfootnote}}
\ead{jhryu034@gmail.com}
\author{Sungsoo Park\fnref{myfootnote}\corref{mycorrespondingauthor}}
\cortext[mycorrespondingauthor]{Corresponding author. Tel +82 42 350 3121}
\ead{sspark@kaist.ac.kr}
\address{Department of Industrial and Systems Engineering, KAIST, 291 Daehak-ro, Yuseong-gu, Daejeon 34141, Republic of Korea}
\fntext[myfootnote]{Korea Advanced Institute of Science and Technology, Daejeon, the Republic of Korea}

\begin{abstract}
We consider the robust single-source capacitated facility location problem with uncertainty in customer demands. A cardinality-constrained uncertainty set is assumed for the robust problem. To solve it efficiently, we propose an allocation-based formulation derived by Dantzig-Wolfe decomposition and a branch-and-price algorithm. The computational experiments show that our branch-and-price algorithm outperforms CPLEX in many cases, which solves the ordinary robust reformulation. We also examine the trade-off relationship between the empirical probability of infeasibility and the additional costs incurred and observe that the robustness of solutions can be improved significantly with small additional costs. 
\end{abstract}


\end{frontmatter}

%\linenumbers

%%%%%%%%%%%%%%%%%%%%%%%%%%%%%%%%%%%%%%%%%%%%%%%%%%%%%%%%
\section{Introduction} \label{section1} 
%%%%%%%%%%%%%%%%%%%%%%%%%%%%%%%%%%%%%%%%%%%%%%%%%%%%%%%%

Facility location problem is one of the important combinatorial optimization problems arising in telecommunication, production-distribution systems, transportation, and many other industrial fields. The problem involves determining optimal locations of facilities and assignments of customers to the facilities with the least cost. A number of variations of the problem and their solution methods have been considered in the literature. Recently, facility location problems under parameter uncertainty have also been addressed and solved by many researchers.   

We consider the single-source capacitated facility location problem (SSCFLP). In this problem, it is only allowed that each customer must be assigned to exactly one facility. Additionally, each facility has a capacity restriction so that it can serve a set of customers as long as the total demand of the assigned customers is within the capacity limit. The objective is to minimize the overall costs of opening the facilities and the assignment of customers to facilities.

The SSCFLP is strongly NP-hard \citep{Cornuejols91, Gadegaard18}, which means there exist neither a pseudo-polynomial time algorithm nor a fully polynomial-time approximation scheme to solve it efficiently unless P=NP. Polynomial-time reduction from the uncapacitated facility location problem \citep{Cornuejols91} or the node cover problem \citep{Gadegaard18} shows this negative theoretical result. However, many algorithms for the SSCFLP have been proposed, which usually fall into one of the categories of Lagrangian relaxation based-algorithms, heuristic algorithms, and branch-and-bound-based exact algorithms.

Lagrangian relaxation has been used, combined with branch-and-bound or heuristics, to obtain lower bounds on the optimal value. \citet{Klincewicz86} proposed Lagrangian relaxation whose relaxed problems are uncapacitated facility location problems by dualizing the capacity constraints, and feasible solutions were obtained by Lagrangian heuristics. \citet{Barcelo84} presented a two-stage algorithm, whose first stage determines facility locations based on Lagrangian relaxation with the single-sourcing constraints relaxed and second stage solves generalized assignment problems. \citet{Pirkul87} and  \citet{Sridharan93} also relaxed the single-sourcing constraints for their Lagrangian relaxation-based algorithm, whose subproblems are binary knapsack problems. \citet{Beasley93} proposed Lagrangian relaxation, dualizing both of the capacity constraints and single-sourcing constraints. \citet{Hindi99} applied greedy heuristics using Lagrangian relaxation and restricted neighborhood search to find solutions of the large-scale SSCFLP. 

Heuristics have also been proposed to obtain high-quality feasible solutions of the SSCFLP in a short time. \citet{Ronnqvist99} identified feasible solutions from a repeated-matching algorithm based on three sets of closed facilities, unassigned customers, and pairs of each assigned customer and her facility, respectively. \citet{Delmaire99} presented a hybrid heuristic algorithm with a greedy randomized adaptive search procedure (GRASP) and tabu search for the SSCFLP. \citet{Cortinhal03} incorporated tabu search into the procedures of Lagrangian heuristics. \citet{Ahuja04} improved multi-exchange heuristics by exchanging the set of customers assigned to each facility. Furthermore, ant colony optimization by \citet{Chen08}, scatter search by \citet{Contreras08}, and kernel search by \citet{Guastaroba14} are proposed for the large-scale SSCFLP.

There have also been studies to solve the SSCFLP exactly by applying branch-and-bound-based algorithms. \citet{Neebe83} formulated the SSCFLP as a set partitioning problem and solved its linear programming relaxation (LP-relaxation) by a column generation approach at each node of the search tree. \citet{Holmberg99} utilized Lagrangian relaxation, relaxing single-sourcing constraints, to obtain lower bounds in the branch-and-bound algorithm. \citet{Diaz02} applied a branch-and-price algorithm with two levels of nodes; The first level nodes, children of the root node, are corresponding to the selection of opened facilities, and the second level nodes, children of the first level nodes, are corresponding to the allocation of customers, respectively. Recently, \citet{Yang12} proposed a modified branch-and-cut algorithm with lifted cover inequalities and Fenchel cutting planes to solve the SSCFLP. Their branching scheme generates a pair of nodes at each level; one involves a small-sized sparse problem with some variables fixed to zero to get feasible solutions, and the other contains a dense problem to obtain lower bounds. \citet{Gadegaard18} improved the algorithm of \citet{Yang12} in terms of cut generations and local branching strategies.  

Meanwhile, there has been much effort to handle facility location problems under uncertainty of parameters such as costs, demands, etc. It has been observed that an optimal solution to a deterministic problem can be inefficient or even infeasible to small changes of problem data \citep{Bental00}. Stochastic programming and robust optimization are two important approaches that have been used to handle parameter uncertainty. 

Stochastic programming is based on the assumption that there are certain probability distributions of all or some of the parameters, which are known in advance. To introduce the overview of the models and solution algorithms for the stochastic facility location problems, we refer to \citet{Owen98} and \citet{Snyder06}. One of the interesting cases of the stochastic facility location problems is demand uncertainty. It has been described using random variables, and therefore, capacity constraints now can be defined as chance-constraints with a probability level. Each chance-constraint states that the probability of the total demand of the customers assigned to a facility exceeding its capacity is less than a specified probability level. \citet{Laporte94} introduced chance-constraints to the capacitated facility location problem with stochastic customer demands. They formulated it as a mixed-integer programming (MIP) problem and solved it by a branch-and-cut algorithm. \citet{Beraldi04} assumed that the demands of emergency medical services follow the Poisson distribution, and they formulated the problem as a stochastic integer programming model with chance-constraints. \citet{Lin09} assumed that the distribution of customer demands of the SSCFLP can be Poisson or normal and defined the capacity restrictions as chance-constraints, which can be formulated as a mixed-integer nonlinear programming problem for the case of normally distributed demand uncertainty.

The additional costs to the objective function incurred by excessive demands at each facility are also considered for the SSCFLP. \citet{Albareda11} provided a formulation of the stochastic SSCFLP with a restriction on the number of assigned customers to each facility when each customer demand, restricted by whether it is necessary or not, follows the Bernoulli distribution. They added the expected value of additional costs, which can occur by reassigning customers to another facility, to the objective function. Then, it can be formulated as an MIP problem, and they solved the instances with at most 20 candidates of facilities and 60 customers by CPLEX. \citet{Bieniek15} extended the assumption on the distribution of demands to arbitrary discrete, continuous, or mixed distributions. The paper includes theoretical results for general distribution and computational experiments for a small instance with four facilities and twelve customers whose demands have exponential or Poisson distribution.   

However, these stochastic programming approaches have some limitations. First, exact distributions of parameters are required for a stochastic programming formulation, but it is not easy to know the true distributions of the parameters practically.  Moreover, even if the probability distributions can be assumed precisely, an optimal solution of the problem often cannot be obtained exactly and effectively by the existing methods. Such difficulty usually comes from non-linearity, sometimes non-convexity of the stochastic objective function, and the chance-constraints in the stochastic MIP problems. 

Robust optimization can be an alternative approach for incorporating the uncertainty of parameters into optimization problems. An uncertainty set, instead of probabilistic information, is used to represent the range of parameter changes for robust optimization problems. For example, there are uncertainty sets such as simple interval uncertainty set \citep{Soyster73}, ellipsoidal uncertainty set \citep{Bental98,Bental00}, and cardinality-constrained uncertainty set \citep{Bertsimas03,Bertsimas04} have been considered. 

There have been several results for robust facility location problems such as \citet{Snyder06_2} and \citet{Gulpinar13}. Moreover, we refer to \citet{Baron11} for a comprehensive review of the robust facility location problems. However, to the best of our knowledge, there has been little previous research for the robust optimization approach for the SSCFLP with demand uncertainty. Recently, \citet{Baron19} proposed the almost robust optimization approach for it. This scenario-based, soft-constrained robust optimization framework allows a solution having a few infeasible scenarios. Their proposed decomposition algorithm could solve instances with at most 25 candidates of facilities and 50 customers in about five minutes. 

Like the deterministic SSCFLP, its robust counterpart also can be reformulated using the Dantzig-Wolfe decomposition, as we suggest in this paper. Because the resulting reformulation has exponentially many variables, it cannot be solved directly.  Column generation and branch-and-price method can be used to solve such a problem. They have been used successfully to solve many difficult combinatorial optimization problems with many variables. We refer to \citet{Barnhart98}, \citet{Desrosiers10}, and \citet{Gamrath10} for further details of the branch-and-price algorithm. 

There have been many successful trials to solve large-scale MIP problems using the branch-and-price algorithm. \citet{Savelsbergh97} solved the generalized assignment problem using a branch-and-price algorithm. \citet{Diaz02} applied a branch-and-price algorithm to solve the deterministic SSCFLP. \citet{Ceselli05} used a branch-and-price algorithm to solve the capacitated p-median problem, which is one of the location problems having the same capacity restrictions and single-source restrictions like the SSCFLP. \citet{Klose07} solved the capacitated facility location problem without the single-source constraints using a branch-and-price algorithm. \citet{Lee12} proposed a branch-and-price algorithm for the robust network design problem without flow bifurcations using the cardinality-constrained uncertainty set for demands. Moreover,  \citet{Lee12_1} also presented a branch-and-price-and-cut algorithm for the robust vehicle routing problem with travel time and demand uncertainty.   

In this paper, we consider the SSCFLP with demand uncertainty using a robust optimization perspective. We assume that the demand of each customer belongs to a specified interval uncertainty set. The cardinality-constrained uncertainty set \citep{Bertsimas03,Bertsimas04} is used to describe the demand uncertainty of the robust SSCFLP. This uncertainty set is less conservative than the simple interval uncertainty set \citep{Soyster73}), and linearity of the formulation can be preserved, unlike the ellipsoidal uncertainty set \citep{Bental98,Bental00}. 

After reformulating the problem using the Dantzig-Wolfe decomposition, we propose a branch-and-price algorithm to solve the robust SSCFLP. We will show how the uncertainty of demands can be isolated into the subproblem in the column generation procedure. Therefore, overall optimization is not affected by the uncertainty of demands. We also consider branching schemes, variable fixing, and early termination to improve the performance of the algorithm. Computational experiments show that the algorithm can solve the robust SSCFLP fast compared to the traditional reformulation approach. Moreover, we make observations by simulation that the robustness of the solutions is improved by incorporating demand uncertainty.

The rest of the paper is organized as follows. In section \ref{section2}, we consider the traditional reformulation of the robust SSCFLP and the Dantzig-Wolfe decomposition-based reformulation. Section \ref{section3} explains the technical details of the branch-and-price algorithm to solve the reformulation. Section \ref{section4} gives computational results of our branch-and-price algorithm compared to the traditional MIP reformulation. Section \ref{section5}  presents the result of the Monte Carlo simulation to show the robustness of the obtained solutions. Finally, Section \ref{section6} summarizes the result of our research.

%%%%%%%%%%%%%%%%%%%%%%%%%%%%%%%%%%%%%%%%%%%%%%%%%%%%%%%%
\section{Formulations of the robust SSCFLP} \label{section2} 
%%%%%%%%%%%%%%%%%%%%%%%%%%%%%%%%%%%%%%%%%%%%%%%%%%%%%%%%

In this section, we introduce MIP formulations of the robust SSCFLP with a cardinality-constrained uncertainty set for demands. We also present an allocation-based formulation that can be obtained using the Dantzig-Wolfe decomposition.   

%%%%%%%%%%%%%%%%%%%%%%%%%%%%%%%%%%%%%%%%%%%%%%%%%%%%%%%%
\subsection{Robust SSCFLP with cardinality-constrained demand uncertainty} \label{section21} 
%%%%%%%%%%%%%%%%%%%%%%%%%%%%%%%%%%%%%%%%%%%%%%%%%%%%%%%%

We first introduce notation as follows. Let $\mathrm{M} = \{1,\cdots,m\}$ be a set of candidate facility locations and $\mathrm{N}=\{1,\cdots,n\}$ be a set of customers. Let $f_{i}$ be the set-up cost of opening facility and $s_{i}$ be the capacity of the facility at location $i \in \mathrm{M}$. Let $d_{j}$ be the demand of customer $j \in \mathrm{N}$ and $c_{ij}$ be the allocation cost of assigning customer $j \in \mathrm{N}$ to facility $i \in \mathrm{M}$. Without loss of generality, we assume that these parameters are nonnegative integers.

Then, we can formulate the SSCFLP as follows:
\begin{align}
\textrm{\textbf{(P)}} \ \ \  \textrm{minimize} \ \ \ & \sum_{i \in \mathrm{M}} \sum_{j \in \mathrm{N}} c_{ij}x_{ij} + \sum_{i \in \mathrm{M}} f_{i} y_{i}  \label{eq01} \\
\textrm{subject to} \ \ \ & \sum_{j \in \mathrm{N}} d_{j}x_{ij}  \leq  s_{i}y_{i}, \ \forall i \in \mathrm{M}, \label{eq02} \\
& \sum_{i \in \mathrm{M}} x_{ij} = 1, \ \forall j \in \mathrm{N}, \label{eq03} \\
& x_{ij}  \leq  y_{i}, \ \forall i \in \mathrm{M}, j \in \mathrm{N}, \label{eq04}\\
& x_{ij} \in \{0,1\}, \ \forall i \in \mathrm{M}, j \in \mathrm{N}, \label{eq05}\\
& y_{i} \in \{0,1\}, \ \forall i \in \mathrm{M}, \label{eq06} 
\end{align}
\noindent
where the binary variable $x_{ij}$ is equal to one if customer $j \in \mathrm{N}$ is served by facility $i \in \mathrm{M}$ and zero otherwise, and the binary variable $y_{i}$ is equal to one if facility $i \in \mathrm{M}$ is opened, and zero otherwise. The objective function (\ref{eq01}) minimizes the sum of total fixed costs and total assignment costs. Constraints (\ref{eq02}) ensure that the total demand of the customers assigned to a facility should not exceed the capacity of the facility.  Constraints (\ref{eq03}) ensure that each customer must be assigned to exactly one facility. Constraints (\ref{eq04}) are redundant, but the lower bound obtained by the LP-relaxation of \textbf{(P)} can be strengthened by adding these constraints.

Now, we formulate the robust SSCFLP under demand uncertainty. A cardinality-constrained uncertainty set \citep{Bertsimas03,Bertsimas04} can be used to describe demand uncertainty. We assume that the demand of each customer $j \in \mathrm{N}$ takes a value in the interval $[d_{j}-b_{j},d_{j}+b_{j}]$, where $d_{j}$ is a nominal demand and $b_{j}$ is a value of the maximum deviation from $d_{j}$. Moreover,  Let $\Gamma_{i}$ be a nonnegative integer parameter that controls the degree of robustness of a solution for each facility $i \in \mathrm{M}$. It restricts the number of uncertain ones among the demands assigned to each facility. In other words, at most $\Gamma_{i}$ demands have their maximum values $d_{j}+b_{j}$ and the others have nominal values $d_{j}$, among the demands assigned to facility $i \in \mathrm{M}$ in the worst-case scenario. Then, the cardinality-constrained uncertainty set \citep{Bertsimas03,Bertsimas04} is defined as $U_{d}^{i} := \{  \bar{d} \in \mathbb{R}_{+}^{n} | \bar{d}_{j}={d}_{j}+{b}_{j} {v}_{j}, \sum_{j \in \mathrm{N}} |{v}_{j}| \leq \Gamma_{i}, |{v}_{j}| \leq 1 \}$ for facility $i \in \mathrm{M}$. 

Here, the capacity constraints (\ref{eq02}) of \textbf{(P)} can be expressed as follows so that the demand uncertainty is reflected using the cardinality-constrained uncertainty set $U_{d}^{i}$.
\begin{align} \label{eq08}
\sum_{j \in \mathrm{N}} \bar{d}_{j}x_{ij}  \leq  s_{i}y_{i}, \ \ \ \forall \bar{d} \in U_{d}^{i}, i \in \mathrm{M}.
\end{align}
These constraints are equivalent to the following nonlinear constraints. 
\begin{align} \label{eq09}
\sum_{j \in \mathrm{N}} d_{j}x_{ij} + \max_{R \subseteq \mathrm{N}, |R|\leq{\Gamma}_{i}} \sum_{j \in R} b_{j}x_{ij} \leq  s_{i}y_{i}, \ \ \ \forall i \in \mathrm{M}. 
\end{align}
We note that these constraints are also equivalent to the following:
\begin{align} \label{eq09_1}
\sum_{j \in \mathrm{N}} d_{j}x_{ij} + \sum_{j \in R} b_{j}x_{ij} \leq  s_{i}y_{i}, \ \ \ \forall R \subseteq \mathrm{N}, |R|\leq{\Gamma}_{i}, i \in \mathrm{M}.
\end{align}
\noindent Then, the problem can be formulated as the following MIP problem by replacing the capacity constraints with constraints (\ref{eq09_1}). 

\begin{align} 
\textrm{\textbf{(RP1)}} \ \ \ \textrm{Minimize} \ \ \ & \sum_{i \in \mathrm{M}} \sum_{j \in \mathrm{N}} c_{ij}x_{ij} + \sum_{i \in \mathrm{M}} f_{i} y_{i} \nonumber \\ 
\textrm{subject to} \ \ \ & \ (\ref{eq03})-(\ref{eq06}), (\ref{eq09_1}). \nonumber
\end{align}

Moreover, we can obtain an alternative MIP formulation of \textbf{(RP1)}. \citet{Bertsimas03, Bertsimas04} 
showed that constraints (\ref{eq09}) can be reformulated using strong duality to the inner maximization term as follows:
\begin{align}
 \sum_{j \in \mathrm{N}} d_{j}x_{ij} + \sum_{j \in \mathrm{N}} p_{ij} + \Gamma_{i} q_{i} \leq  s_{i}y_{i}, \ \ \ \forall i \in \mathrm{M}, \label{eq10} \\
 q_{i} + p_{ij} \geq b_{i}x_{ij},  \ \ \ \forall i \in \mathrm{M}, j \in \mathrm{N}, \label{eq11} \\
 p_{ij} \geq 0, \ \ \ \forall i \in \mathrm{M}, j \in \mathrm{N}, \label{eq12} \\
 q_{i} \geq 0, \ \ \ \forall i \in \mathrm{M} \label{eq13}.
\end{align} 
\noindent Therefore, the robust SSCFLP can be reformulated as the following MIP problem:
\begin{align} \nonumber
\textrm{\textbf{(RP2)}} \ \ \ \textrm{minimize} \ \ \ & \sum_{i \in \mathrm{M}} \sum_{j \in \mathrm{N}} c_{ij}x_{ij} + \sum_{i \in \mathrm{M}} f_{i} y_{i}  \\ \nonumber
\textrm{subject to} \ \ \ & (\ref{eq03})-(\ref{eq06}),  \\ \nonumber
& (\ref{eq10})-(\ref{eq13}).   \nonumber
\end{align}

\textbf{(RP2)} is a mixed-integer programming problem with $mn+m$ additional variables and $2mn+2m$ additional constraints than problem \textbf{(P)}, which makes it more difficult to solve than the deterministic problem \textbf{(P)}. We also note that the lower bounds obtained by solving the LP-relaxation of \textbf{(RP1)} and LP-relaxation of \textbf{(RP2)} are the same. Instead of solving \textbf{(RP2)}, we propose an allocation-based formulation that isolates the difficulty arising from the demand uncertainty into the subproblem. The allocation based formulation provides a stronger LP relaxation bound than the LP relaxation of \textbf{(RP2)}. 

%%%%%%%%%%%%%%%%%%%%%%%%%%%%%%%%%%%%%%%%%%%%%%%%%%%%%%%%
\subsection{Allocation-based formulation of the robust SSCFLP} \label{section22} 
%%%%%%%%%%%%%%%%%%%%%%%%%%%%%%%%%%%%%%%%%%%%%%%%%%%%%%%%

We derive an allocation-based formulation of the robust SSCFLP by employing the Dantzig-Wolfe decomposition. We take (\ref{eq03}) as the joint constraints, and the other constraints of \textbf{(P)} with the robust capacity constraints (\ref{eq09}) substituting the ordinary capacity constraints (\ref{eq02}) can be used for decomposition. Let $\Omega_{i} := \{ (x_{i1},\dots,x_{in},y_{i}) \in \{0,1\}^{n+1} | \sum_{j \in \mathrm{N}} d_{j}x_{ij} + \max_{R \subseteq \mathrm{N}, |R| \leq \Gamma} \sum_{j \in R} b_{j}x_{ij} \leq  s_{i}y_{i} \}$ for $i \in \mathrm{M}$. Then, $\Omega_{i}$ consists of $(x_{i}^{R},0)=(\textbf{0},0)$ and $(x_{i}^{R},1)$ for $R \in \mathrm{S}^{i}$, where $x_{i}^{R}$ is the incidence vector of set $R \subseteq \mathrm{N}$ of customers and $\mathrm{S}^{i}$ is a set of all possible allocations of customers to facility $i \in \mathrm{M}$, i.e. 
\begin{align} \label{eq14}
 \mathrm{S}^{i} := \{ R \subseteq \mathrm{N} | \sum_{j \in R} d_{j} + \max_{S \subseteq R, |S|\leq{\Gamma}_{i}} \sum_{j \in S} b_{j} \leq s_{i} \}, \quad \forall i \in \mathrm{M}. & &
\end{align}
We can define associated binary variables $v_{i}$ for $(\textbf{0},0)$ and binary variables $z_{R}^{i}$ for $(x_{i}^{R},1)$ for $R \in \mathrm{S}^{i}$, $i \in \mathrm{M}$, respectively, with constraints $\sum_{R \in \mathrm{S}^{i}} z_{R}^{i} + v_{i} = 1$ for $i \in \mathrm{M}$. Then, the binary variables $x_{ij}$ and $y_{i}$ can be expressed as $x_{ij}=\sum_{R \in \mathrm{S}^{i}} x_{ij}^{R} z_{R}^{i}=\sum_{R \in \mathrm{S}^{i} ; j \in R} z_{R}^{i}$ and $y_{i} = 0 \cdot v_{i} + \sum_{R \in \mathrm{S}^{i}} 1 \cdot z_{R}^{i} = \sum_{R \in \mathrm{S}^{i}} z_{R}^{i}$, respectively. 

Now, substituting for $x_{ij}$ and $y_{i}$ variables in constraints (\ref{eq03}) and the objective function (\ref{eq01}) leads to the reformulation of the robust SSCFLP as follows:
\begin{align} 
\textrm{\textbf{(AP)} \ Minimize} \ \ \ & \sum_{i \in \mathrm{M}} \sum_{R \in \mathrm{S}^{i}} c_{R}^{i} z_{R}^{i} & \label{eq15} \\
\textrm{subject to} \ \ \ & \sum_{i \in \mathrm{M}} \sum_{R \in \mathrm{S}^{i} ; j \in R } z_{R}^{i} = 1, \ \ \  \forall j \in \mathrm{N}, \label{eq16} \\
& \sum_{R \in \mathrm{S}^{i}} z_{R}^{i} + v_{i} = 1, \ \ \ \forall i \in \mathrm{M}, \label{eq17} \\
& v_{i} \in \{0,1\}, \ \ \  \forall i \in \mathrm{M}, \label{eq18_1} \\
& z_{R}^{i} \in \{0,1\}, \ \ \  \forall R \in \mathrm{S}^{i}, i \in \mathrm{M}, \label{eq18} 
\end{align}
where $c_{R}^{i} := \sum_{j \in R} c_{ij} +  f_{i}$ for $R \in \mathrm{S}^{i}$, $i \in \mathrm{M}$. The binary variable $v_{i}$ is equal to one if facility $i \in \mathrm{M}$ is not opened, and zero otherwise. Also, The binary variable $z_{R}^{i}$ is equal to one if opened facility $i$ covers customers in $R$, and zero otherwise for $R \in \mathrm{S}^{i}$, $i \in \mathrm{M}$. The objective function (\ref{eq15}) minimizes the total fixed costs and assignment costs of all opened facilities. Constraints (\ref{eq16}) ensure that all customers must be covered by exactly one combination of customers at each facility. Constraints (\ref{eq17}) ensure that each facility must be closed or opened, and it must take exactly one combination of customers when it is opened.

% (If you want to learn more about the Dantzig-Wolfe decomposition, you can read the textbook \citet{Wolsey98}.) 

%%%%%%%%%%%%%%%%%%%%%%%%%%%%%%%%%%%%%%%%%%%%%%%%%%%%%%%%
\section{Branch-and-price algorithm} \label{section3} 
%%%%%%%%%%%%%%%%%%%%%%%%%%%%%%%%%%%%%%%%%%%%%%%%%%%%%%%%

In this section, we present a branch-and-price algorithm for the allocation-based formulation \textbf{(AP)} of the robust SSCFLP. 

%%%%%%%%%%%%%%%%%%%%%%%%%%%%%%%%%%%%%%%%%%%%%%%%%%%%%%%%
\subsection{Linear programming master problem} \label{section31} 
%%%%%%%%%%%%%%%%%%%%%%%%%%%%%%%%%%%%%%%%%%%%%%%%%%%%%%%%

The LP-relaxation of the allocation-based formulation \textbf{(AP)} can be obtained by dropping the integrality restrictions on the variables as follows: 
\begin{align} 
\textrm{Minimize} \ \ \ & \sum_{i \in \mathrm{M}} \sum_{R \in \mathrm{S}^{i}} c_{R}^{i} z_{R}^{i} & \nonumber \\
\textrm{subject to} \ \ \ & \sum_{i \in \mathrm{M}} \sum_{R \in \mathrm{S}^{i} ; j \in R } z_{R}^{i} = 1, \ \ \ \forall j \in \mathrm{N}, \label{eq20} \\
& \sum_{R \in \mathrm{S}^{i}}z_{R}^{i} + v_{i} = 1, \ \ \ \forall i \in \mathrm{M}, \label{eq21} \\
& 0 \leq  v_{i} \leq 1, \ \ \ \forall i \in \mathrm{M}. \label{eq22_1} \\
& 0 \leq z_{R}^{i} \leq 1, \ \ \ \forall R \in \mathrm{S}^{i}, i \in \mathrm{M},  \label{eq22} 
\end{align}

We can compare the strength of the LP-relaxation of \textbf{(AP)} and LP-relaxation of \textbf{(RP1)} and its reformulation \textbf{(RP2)}. 

\begin{proposition} \label{prop01}
The LP-relaxation of \textbf{(AP)} has the same optimal value as that of the LP-relaxation of \textbf{(RP1)} augmented with all valid inequalities describing the convex hull of $\Omega_{i}$, $i \in M$.
\end{proposition} 
\begin{proof}
The LP-relaxation of \textbf{(AP)} can be obtained by substituting $x_{ij}=\sum_{R \in \mathrm{S}^{i} ; j \in R} z_{R}^{i}$, $y_{i} = \sum_{R \in \mathrm{S}^{i}} z_{R}^{i}$, and constraints (\ref{eq21})-(\ref{eq22}). This is equivalent to substituting $(x_{i1}, \cdots, x_{in}, y_{i}) \in \textrm{conv}(\Omega_{i})$.
\end{proof}

\noindent Because the lower bounds obtained by solving the LP-relaxation of \textbf{(RP1)} and LP-relaxation of \textbf{(RP2)} are the same, we can see that the LP-relaxation of \textbf{(AP)} provides a stronger lower bound than that of the LP-relaxation of \textbf{(RP2)}. We also note that the LP-relaxation bound of \textbf{(AP)} is the same as the Lagrangian dual bound when the joint constraints (\ref{eq03}) are dualized.

The LP-relaxation of \textbf{(AP)} can be modified further for improving computational efficiency.  Constraints (\ref{eq20}) can be replaced by inequalities:
\begin{align} \label{eq23} 
& \sum_{i \in \mathrm{M}} \sum_{R \in \mathrm{S}^{i} ; j \in R } z_{R}^{i} \geq 1, \ \ \ \forall j \in \mathrm{N}.
\end{align} 
\noindent
Because all set-up costs and allocation costs are nonnegative, there exists an optimal solution that also satisfies constraints (\ref{eq20}). Constraints (\ref{eq23}) restrict the corresponding dual variables to be nonnegative, which can make the column generation procedure more stable compared to using unrestricted dual variables. 

Constraints (\ref{eq21}) can be replaced by inequalities:
\begin{align} \label{eq24}
& -\sum_{R \in \mathrm{S}^{i}} z_{R}^{i} \geq -1, \ \ \ \forall i \in \mathrm{M}, 
\end{align} 
because the variables $v_{i}$ can be regarded as slack variables in constraints (\ref{eq21}). In addition to this, constraints (\ref{eq22}) can be replaced by inequalities:
\begin{align} \label{eq25}
& z_{R}^{i} \geq 0, \ \ \ \forall R \in \mathrm{S}^{i}, i \in \mathrm{M}, 
\end{align} 
because of constraints (\ref{eq24}).

As a result, the LP-relaxation of \textbf{(AP)} can be stated as the following linear programming master problem:
\begin{align} 
\textbf{(MP)} \ \ \ \textrm{Minimize} \ \ \ & \sum_{i \in \mathrm{M}} \sum_{R \in \mathrm{S}^{i}} c_{R}^{i} z_{R}^{i} \ \ \ \nonumber \\
\textrm{subject to} \ \ \ & (\ref{eq23}), (\ref{eq24}), (\ref{eq25}).  \nonumber
\end{align}

%%%%%%%%%%%%%%%%%%%%%%%%%%%%%%%%%%%%%%%%%%%%%%%%%%%%%%%%
\subsection{Restricted master problem and subprolem} \label{section32} 
%%%%%%%%%%%%%%%%%%%%%%%%%%%%%%%%%%%%%%%%%%%%%%%%%%%%%%%%

We cannot solve \textbf{(MP)} directly since it has exponencially many variables. We suppose that we have a subset $\mathrm{R}^{i}$ of $\mathrm{S}^{i}$ for $i \in \mathrm{M}$ which provides a feasible solution to \textbf{(MP)}. Then, the following restricted problem \textbf{(RMP)} can be obtained: 
\begin{align} 
\textbf{(RMP)} \ \ \ \textrm{Minimize} & \sum_{i \in \mathrm{M}} \sum_{R \in \mathrm{R}^{i}} c_{R}^{i} z_{R}^{i}  \nonumber \\
\textrm{subject to} & \sum_{i \in \mathrm{M}} \sum_{R \in \mathrm{R}^{i} ; j \in R } z_{R}^{i} \geq 1, \ \ \ \forall j \in \mathrm{N}, \label{eq26} \\
& -\sum_{R \in \mathrm{R}^{i}}z_{R}^{i} \geq -1, \ \ \ \forall i \in \mathrm{M}. \label{eq27} \\
& z_{R}^{i} \geq 0, \ \ \ \forall R \in \mathrm{R}^{i}, i \in \mathrm{M}.  \label{eq28}
\end{align}

\noindent We solve \textbf{(RMP)} by the simplex method and obtain an optimal solution $z^{*}$ with optimal value $\underline{Z}$. Let $\lambda \in \mathbb{R}_{+}^{n}$ and $\mu \in \mathbb{R}_{+}^{m}$ be a dual optimal solution corresponding to constraints (\ref{eq26}) and (\ref{eq27}), respectively. 

During column generation, a column with a negative reduced cost is generated and added to $\textbf{(RMP)}$ iteratively. This procedure continues until an optimal solution of $\textbf{(RMP)}$ becomes also optimal for $\textbf{(MP)}$. The reduced cost of a variable $z_{R}^{i}$ is $\mathrm{RC}_{i, R}(\lambda,\mu) := \sum_{j \in R}(c_{ij}-\lambda_{j})+f_{i}+\mu_{i}$ for each $ R \in \mathrm{S}^{i}, i \in \mathrm{M}$. We then try to find a column having negative reduced cost by solving the following subproblem:
\begin{align} \nonumber
\textrm{\textbf{(Sub-i)}} \ \ \ \textrm{Maximize} \ \ \ & \xi^{i} := \sum_{j \in \mathrm{N}}(\lambda_{j}-c_{ij})x_{ij} &  \\ \nonumber
\textrm{subject to} \ \ \ & d_{j}x_{ij} + \max_{R \subseteq \mathrm{N}, |R| \leq {\Gamma}_{i}} \sum_{j \in R} b_{j}x_{ij} \leq  s_{i}, & \nonumber \\
& x_{ij} \in \{0,1\}, \ \ \ \forall j \in \mathrm{N}. \nonumber
\end{align} 
This problem is the robust binary knapsack problem with cardinality-constrained weight uncertainty. If every 
\begin{align} \nonumber
\min_{R \in \mathrm{S}^{i}} \mathrm{RC}_{i, R}(\lambda,\mu) = -\xi^{i}+f_{i}+\mu_{i}
\end{align} 
has a nonnegative value for $i \in \mathrm{M}$, an optimal solution of \textbf{(RMP)} is also an optimal solution of \textbf{(MP)}. Otherwise, if $\mathrm{RC}_{i, R}(\lambda,\mu)<0$, an optimal solution of \textbf{(Sub-i)} generates a column which has the smallest negative reduced cost among columns involving facility $i$ for $i \in \mathrm{M}$.

\citet{Bertsimas03} showed that the robust BKP can be solved by solving the ordinary BKPs at most $n+1$ times. \citet{Lee12} reduced the number of iterations to at most $n-\Gamma_{i}+1$ times, and we apply it to solve \textbf{(Sub-i)}. Let $C^{i}$ be a feasible solution set of \textbf{(Sub-i)} for $i \in \mathrm{M}$. We assume that the values of the maximum possible deviation from the nominal demand are listed in nonincreasing order, and define a dummy value $b_{n+1} = 0$, i.e. $b_{1} \geq b_{2} \geq \dots \geq b_{n} \geq b_{n+1} = 0$. We define a set $\mathrm{N}^{+} = \mathrm{N} \cup \{n+1\}$ and sets $N_{l}=\{1, \dots l\}$ for all $l \in \mathrm{N}^{+}$. We then define $C_{l}^{i}=\{x_{i} \in \{0,1\}^{n}| \sum_{j \in \mathrm{N}} d_{j}x_{ij} + \sum_{j \in N_{l} \cap \mathrm{N}} (b_{j}-b_{l})x_{ij} \leq  s_{i}-\Gamma_{i}b_{l}\}$ for $l \in \{\Gamma_{i}, \Gamma_{i}+1, \dots, n-1, n+1 \}$. Then, $C^{i}$ can be obtained using the solution sets of ordinary binary knapsack problems (BKP). 
\begin{proposition} \label{prop02}
$C^{i}=\cup_{l \in \{\Gamma_{i},\Gamma_{i}+1, \dots, n-1, n+1 \}} C_{l}^{i}$
\end{proposition}
\begin{proof}
We refer to \citet{Lee12} for the proof. 
\end{proof}
\noindent Proposition \ref{prop02} implies that we can solve \textbf{(Sub-i)} by solving BKPs $n-\Gamma_{i}+1$ times and taking the best solution among the optimal solutions to BKPs. 

The BKP can be solved by a branch-and-bound algorithm or a dynamic programming approach. \citet{Pisinger97} has provided the minknap algorithm based on the dynamic programming with pseudo-polynomial time complexity of $O(ns_i)$. Moreover, \citet{Martello99} showed that the minknap algorithm solved the BKP faster than the other algorithms based on a branch-and-bound algorithm only. As mentioned previously, we solve the RBKP by solving the BKP $n-\Gamma+1$ times, and it has pseudo-polynomial time complexity of $O((n-\Gamma_{i}+1)ns_i)$. 

We note that we may solve the MIP reformulation of the RBKP using Bertsimas and Sim's approach \citep{Bertsimas03,Bertsimas04}. However, \citet{Monaci13} reported that the algorithm of \citet{Lee12} solved the RBKP more effectively than CPLEX, which solved the MIP reformulation of the RBKP.

%%%%%%%%%%%%%%%%%%%%%%%%%%%%%%%%%%%%%%%%%%%%%%%%%%%%%%%%
\subsection{Branching scheme} \label{section33} 
%%%%%%%%%%%%%%%%%%%%%%%%%%%%%%%%%%%%%%%%%%%%%%%%%%%%%%%%

If the optimal solution to \textbf{(MP)} has fractional values, we need to branch. However, direct branching on $z_{R}^{i}$ variables is not desirable. For example, if we branch on a variable $z_{R}^{i}$, two nodes are generated; one has $z_{R}^{i}=0$, and the other has $z_{R}^{i}=1$. If $z_{R}^{i}$ is fixed to zero, we need to make sure that the column for $z_{R}^{i}$ will not be generated again in subsequent column generation procedure, which is a nontrivial task. Such branching scheme also divides the feasible solution set unevenly. \citet{Diaz02} discuss this defect of branching on the $z_{R}^{i}$ variables directly.  

Instead, we use branching on the variables of \textbf{(RP1)} directly as suggested in \citet{Ceselli05}. Let $z^{*}$ be an optimal solution of \textbf{(MP)}. Then, the value of $x$ and $y$ variables can be obtained as $x_{ij}^{*} = \sum_{R \in \mathrm{S}^{i} ; j \in R } z_{R}^{i*}$ for $i \in \mathrm{M}, j \in \mathrm{N},$ and $y_{i}^{*} = \sum_{R \in \mathrm{S}^{i}} z_{R}^{i*}$ for $i \in \mathrm{M}$ as shown in section \ref{section22}. We note that $x$ and $y$ variables are integral if and only if $z$ variables are integral.  

\citet{Ceselli05} used branching on $x$ variables only for the capacitated p-median problem, which is a variation of facility location problem. However, our preliminary testing showed that branching on $y$ variables first and then on $x$ variables gives better results. Therefore, we do branching on $x$ variables when all $y$ variables have integer values.  \citet{Holmberg99} also discussed some advantages and disadvantages of each branching scheme for the SSCFLP.  

When we branch on $y$ variables, we branch on the variable $y_{i}$ having value closest to $0.5$ among the candidate $y$ variables for branching. We set $y_{i}=0$ on one branch, and $y_{i}=1$ on the other branch.    

If all $y$ variables are integer-valued and there are some fractional $x$ variables, we branch on $x$ variables. Let $x^{*}$ be the current fractional solution. We identify a customer $j'$ and use generalized upper bound (GUB) dichotomy \citep{Savelsbergh97} on the variables $x_{ij'}$ for all $i \in \mathrm{M}$. 

Let $N^{*} \subseteq \mathrm{N}$ be the set of customers $j$ such that $x_{ij}>0$ for more than one $i \in \mathrm{M}$. For each $j \in N^{*}$, we divide $\mathrm{M}$ into four sets $M^{11}_{j}$, $M^{12}_{j}$, $M^{21}_{j}$, and $M^{22}_{j}$ as follows. First, we divide $\mathrm{M}$ into disjoint sets $M^{1}_{j}$ and $M^{2}_{j}$ for each $j \in \mathrm{N}$, such that $i \in M^{1}_{j}$ if $x_{ij}^{*}>0$, and $i \in M^{2}_{j}$ otherwise. Second, $M^{1}_{j}$ is divided into disjoint sets $M^{11}_{j}$ and $M^{12}_{j}$ such that $M^{11}_{j}$ minimizes $\lvert \sum_{i \in M^{11}_{j}} x_{ij}^{*} - 0.5 \rvert$. We solve the following ordinary knapsack problem:
\begin{align} \nonumber
\kappa_{j} := \textrm{max}_{ M^{11}_{j} \subseteq M^{1}_{j}} \{ \sum_{i \in M^{11}_{j}} x_{ij}^{*} \ | \ \sum_{i \in M^{11}_{j}} x_{ij}^{*} \leq 0.5 \}, \ \ \ j \in N^{*},
\end{align} 
to divide set $M^{1}_{j}$ for each $j \in N^{*}$. Third, choose customer $j'=\textrm{argmin}_{j \in N^{*}} \lvert \kappa_{j} - 0.5 \rvert$ to make $\sum_{i \in M^{11}_{j'}} x_{ij'}^{*}$ close to $0.5$. Finally, divide $M^{2}_{j'}$ into disjoint sets $M^{21}_{j'}$ and $M^{22}_{j'}$ to be of the same size. We then branch on $x_{ij'}$ for all $i \in \mathrm{M}$. We set $x_{ij'}=0$ for $i \in M^{11}_{j'} \cup M^{21}_{j'}$ on one branch, and we set $x_{ij'}=0$ for $i \in M^{12}_{j'} \cup M^{22}_{j'}$ on the other. 

We need to reflect the effect of some fixed variables to \textbf{(RMP)} and modify subproblem during subsequent column generation procedure. If variable $y_{i}$ is fixed to zero i.e. $y_{i}=0$, then we set the upper bounds of $z_{R}^{i}$ variables to zero for all $R \in \mathrm{R}^{i}$, and we do not solve \textbf{(Sub-i)} during the column generation. Meanwhile, if variable $y_{i}$ is fixed to one, i.e. $y_{i}=1$, inequality of constraints (\ref{eq27}) of \textbf{(RMP)} is replaced by equality. As a result, the dual variable $\mu_{i}$ becomes free without nonnegativity, but other constraints of \textbf{(RMP)} are not changed and \textbf{(Sub-i)} still remains to be the robust BKP.
 
If variable $x_{ij}$ is fixed to zero, i.e. $x_{ij}=0$, then we fix the upper bounds of $z_{R}^{i}$ variables to zero for all $R \in \mathrm{R}^{i}$ satisfying $j \in R$. Also, the subproblem \textbf{(Sub-i)} does not generate the column with $x_{ij}=1$ during the column generation by setting the objective coefficient of $x_{ij}$ to some negative value. We note that we do not need to consider the case of fixing variable $x_{ij}$ to one because GUB dichotomy has been adopted for branching on $x$ variables.

%%%%%%%%%%%%%%%%%%%%%%%%%%%%%%%%%%%%%%%%%%%%%%%%%%%%%%%%
\subsection{Early termination and variable fixing} \label{section34} 
%%%%%%%%%%%%%%%%%%%%%%%%%%%%%%%%%%%%%%%%%%%%%%%%%%%%%%%%

An optimal dual solution of \textbf{(RMP)} can be used to facilitate the branch-and-price procedure. In this section, we consider how the column generation can be terminated earlier before \textbf{(MP)} is completely optimized and how to fix the values of some variables. 

The dual problem of \textbf{(MP)} with the dual variables $\lambda \in \mathbb{R}_{+}^{n}$ and $\mu \in \mathbb{R}_{+}^{m}$ is as follows:
\begin{align} 
\textbf{(DMP)} \ \ \ \textrm{Maximize} \ \ \ & \sum_{j \in \mathrm{N}} \lambda_{j} - \sum_{i \in \mathrm{M}}  \mu_{i} & \nonumber \\
\textrm{subject to} \ \ \ &  \sum_{j \in R} \lambda_{j} - \mu_{i} \leq c_{R}^{i}, \ \ \ \forall R \in \mathrm{S}^{i}, i \in \mathrm{M},  \label{eq30} \\
& \lambda_{j} \geq 0, \ \ \ \forall j \in \mathrm{N}, \nonumber \\
& \mu_{i} \geq 0, \ \ \ \forall i \in \mathrm{M},  \nonumber
\end{align}
where $c_{R}^{i} = f_{i} + \sum_{j \in R} c_{ij}$ for $R \in \mathrm{S}^{i}$, $i \in \mathrm{M}$. The dual of the restricted master problem \textbf{(RMP)} can be obtained by substituting $\mathrm{S}^{i}$ by $\mathrm{R}^{i}$ in constraints (\ref{eq30}). We call it \textbf{(DRMP)}.   

Let $(\lambda^{*}, \mu^{*})$ be an optimal solution of \textbf{(DRMP)}. We note that the minimum reduced cost of variables $z_{R}^{i}$, $R \in \mathrm{S}^{i}$ for each $i \in \mathrm{M}$ is equal to $\mu_{i}^{*} + f_{i}-\xi^{i*}$, where $\xi^{i*}$ is optimal value of \textbf{(Sub-i)} with $\lambda=\lambda^{*}$. Let $\nu^{*} \in \mathbb{R}^{m}$ be a vector where $\nu_{i}^{*} := \mu^{*}_{i} + \mathrm{min} \{f_{i}-\xi^{i*}, 0\}$. Then, $(\lambda, \mu) = (\lambda^{*}, \mu^{*}-\nu^{*})$ is a feasible solution to \textbf{(DMP)}. 

\begin{proposition} \label{prop03}
$(\lambda, \mu) = (\lambda^{*}, \mu^{*}-\nu^{*})$ is a feasible solution to \textbf{(DMP)}.
\end{proposition}
\begin{proof}
For each $i \in \mathrm{M}$, $\mu_{i}^{*}-\nu_{i}^{*}=-\mathrm{min} \{f_{i}-\xi^{i*}, 0\} \geq 0$. Also, constraints (\ref{eq30}) are equivalent to $-\mu_{i} \leq \textrm{min}_{R \in \mathrm{S}^{i}} (c_{R}^{i} - \sum_{j \in R} \lambda_{j})= f_{i} + \textrm{min}_{R \in \mathrm{S}^{i}} \sum_{j \in R} (c_{R}^{i} - \lambda_{j}) = f_{i} - \xi^{i}$ for $i \in \mathrm{M}$. We can see that the nonnegative vector $(\lambda, \mu)$ with $\lambda = \lambda^{*}$ and $\mu_{i} = -\mathrm{min} \{f_{i}-\xi^{i*}, 0\}$, $i \in \mathrm{M}$ satisfies these constraints obviously. Therefore, $(\lambda, \mu) = (\lambda^{*}, \mu^{*}-\nu^{*})$ is a feasible solution to \textbf{(DMP)}.
\end{proof}

As a result, the objective value $\sum_{j \in \mathrm{N}} \lambda_{j}^{*} - \sum_{i \in \mathrm{M}} (\mu_{i}^{*}-\nu_{i}^{*})= \underline{Z} + \sum_{i \in \mathrm{M}} \nu_{i}^{*}$ can provide a lower bound on the optimal value of \textbf{(MP)}, where $\underline{Z}$ is the optimal value to \textbf{(RMP)}. If $\underline{Z} + \sum_{i \in \mathrm{M}} \nu_{i}^{*}$ is greater than the current incumbent value, the column generation is terminated and the node is pruned. 

Moreover, we can fix the value of some $x$ and $y$ variables to reduce the solution space although only $z$ variables appear in \textbf{(MP)} and \textbf{(RMP)}. The reduced costs of $x$ and $y$ variables can be computed by adding the constraints $-x_{ij} + \sum_{R \in \mathrm{S}^{i} ; j \in R} z_{R}^{i}=0$, $x_{ij} \geq 0$ for $i \in \mathrm{M}$, $j \in \mathrm{N}$ and $-y_{i} + \sum_{R \in \mathrm{S}^{i}} z_{R}^{i} = 0$, $y_{i} \geq 0$ for $i \in \mathrm{M}$ to \textbf{(MP)}, respectively, when they have zero values. Also, the reduced cost of surplus $v$ variables, where $v_{i} = 1- \sum_{R \in \mathrm{S}^{i}} z_{R}^{i} = 1-y_{i}$, $i \in \mathrm{M}$, also can be computed by adding the constraints $-v_{i} - \sum_{R \in \mathrm{S}^{i}} z_{R}^{i} = -1$, $v_{i} \geq 0$, for $i \in \mathrm{M}$ to \textbf{(MP)}. We note that fixing $v_i$ to zero is equivaluent to fixing $y_i$ to one. Let $\delta \in \mathbb{R}^{m \times n}$, $\rho \in \mathbb{R}^{m}$, and $\tau \in \mathbb{R}^{m}$ denote an optimal dual solution corresponding to the coupling constraints for $x$, $y$, and $v$ variables, respectively. Because only these coupling constraints have $x$, $y$, and $v$ variables, the reduced cost of $x_{ij}$ is equal to $0 - (-1) \delta_{ij} = \delta_{ij}$ for $i \in \mathrm{M}$, $j \in \mathrm{N}$, the reduced cost of $y_{i}$ is equal to $0 - (-1) \rho_{i} = \rho_{i}$, and the reduced cost of $v_{i}$ is equal to $0 - (-1) \tau_{i} = \tau_{i}$. 

This technique is based on the approach of \citet{Dearagao03} and \citet{Fukasawa06}. However, explicitly adding these constraints in the problems not only increases the size of the problem but changes the structure of the master problem and the subproblem. Moreover, these explicit coupling constraints can amplify degeneracy problem of dual feasible solutions, which can hamper the convergence of the column generation procedure \citep{Lee12}. Therefore, we apply an alternative approach without using the explicit coupling constraints. 
 
When $x_{ij}=0$ for some $i \in \mathrm{M}$, $j \in \mathrm{N}$ in an optimal solution to \textbf{(MP)}, we assume that the coupling constraint $-x_{ij}+\sum_{R \in \mathrm{S}^{i} ; j \in R} z_{R}^{i} = 0$ is included in \textbf{(MP)}. Let $\mathrm{S}^{ij}$ be a set of all possible customer allocations including customer $j$ to facility $i$ i.e. $\mathrm{S}^{ij} := \{ R \in \mathrm{S}^{i} \ | \ j \in R \} $. Let $\xi^{ij}$ be the optimal value of \textbf{(Sub-i)} when $x_i$ is fixed to one. It can be solved like the original subproblem because we can solve the ordinary knapsack problem $n-\Gamma_i+1$ times after fixing $x_j=1$ when we solve $\textbf{(Sub-i)}$. The following proposition shows that the reduced cost of $x_{ij}$ can be calculated without adding the additional explicit coupling constraints to \textbf{(MP)}.  

\begin{proposition} \label{prop04}
Let $(\lambda^{*}, \mu^{*})$ be an optimal solution to \textbf{(DMP)}. For given $i \in \mathrm{M}$ and $j \in \mathrm{N}$, let $\delta_{ij}^{*} = -\xi^{ij*}+f_{i}+\mu_{i}^{*}$. Then, $(\lambda^{*}, \mu^{*}, \delta_{ij}^{*})$ is an optimal solution to the dual problem of \textbf{(MP)} with the coupling constraints $-x_{ij}+\sum_{R \in \mathrm{S}^{i} ; j \in R} z_{R}^{i} = 0, x_{ij} \geq 0$. Moreover, the reduced cost of $x_{ij}$ is equal to $\delta_{ij}^{*}$. 
\end{proposition}
\begin{proof}
After augmenting the coupling constraint to \textbf{(MP)}, constraints (\ref{eq30}) for $i \in \mathrm{M}$ are replaced by $ \delta_{ij} + \sum_{k \in R} \lambda_{k} - \mu_{i} \leq c_{R}^{i}$ for $R \in \mathrm{S}^{ij}$, and $\sum_{k \in R} \lambda_{k} - \mu_{i} \leq c_{R}^{i}$ for $R \in \mathrm{S}^{i} \setminus \mathrm{S}^{ij}$. The first constraints are equivalent to $\delta_{ij} \leq \mathrm{min}_{R \in \mathrm{S}^{ij}} \{ - \sum_{k \in R} \lambda_{k} + \mu_{i} + c_{R}^{i} \}$. Its right-hand side is equal to $- \mathrm{max}_{R \in \mathrm{S}^{ij}}$  $\sum_{k \in R}  (\lambda_{k} -c_{ik}) + f_{i} + \mu_{i} = -\xi^{ij}+f_{i}+\mu_{i}$. Also, the dual constraint corresponding to $x_{ij}$ is $-\delta_{ij} \leq 0$, and $\delta_{ij}^{*} \geq -\xi^{i*}+f_{i}+\mu_{i}^{*}$ is feasible to this constraint. Hence, $(\lambda^{*}, \mu^{*}, \delta_{ij}^{*})$ is a dual feasible solution. Because the dual objective function is independent of $\delta_{ij}$,  $(\lambda^{*}, \mu^{*}, \delta_{ij}^{*})$ is an optimal solution to the dual problem of \textbf{(MP)} with the coupling constraints. Thus, $\delta_{ij}^{*}$ is the reduced cost of $x_{ij}$. 
\end{proof}

\noindent Let $\underline{Z}$ be the optimal value of \textbf{(MP)} and $\overline{Z}$ be the value of the currrent incumbent solution to $\textbf{(AP)}$. If $x_{ij}$ is equal to zero with reduced cost $\delta_{ij}^{*}$, and $\underline{Z} + \delta_{ij}^{*}$ is greater than $\overline{Z}$, then there exists an optimal solution to \textbf{(AP)} with $x_{ij}=0$. Therefore, we can fix the value of $x_{ij}$ to zero in subsequent branch-and-price procedure.   

Variable fixing for $y_{i}=0$, and $y_{i}=1$ can be done similarly. However, we present the next two propositions for completeness. When $y_{i}=0$ for some $i \in \mathrm{M}$ in an optimal solution to \textbf{(MP)}, we assume that the coupling constraints $-y_{i} +\sum_{R \in \mathrm{S}^{i}} z_{R}^{i} = 0$, $y_{i} \geq 0$ are included in \textbf{(MP)}. The following proposition shows that the reduced cost of $y_{i}$ can be calculated without adding the additional explicit coupling constraints to \textbf{(MP)}.  

\begin{proposition} \label{prop05}
Let $(\lambda^{*}, \mu^{*})$ be an optimal solution to \textbf{(DMP)}. For given $i \in \mathrm{M}$, let $\rho_{i}^{*} =-\xi^{i*} + f_{i} + \mu^{*}_{i}$. Then, $(\lambda^{*}, \mu^{*}, \rho_{i}^{*})$ is an optimal solution to the dual problem of \textbf{(MP)} with the coupling constraints $-y_{i} + \sum_{R \in \mathrm{S}^{i}} z_{R}^{i} = 0$, $y_{i} \geq 0$. Moreover, the reduced cost of $y_{i}$ is equal to $\rho_{i}^{*}$. 
\end{proposition}
\begin{proof}
After augmenting the coupling constraint to \textbf{(MP)}, constraints (\ref{eq30}) for $i \in \mathrm{M}$ are replaced by $\rho_{i} + \sum_{k \in R} \lambda_{k} - \mu_{i} \leq c_{R}^{i}$ for $R \in \mathrm{S}^{i}$. The constraints are equivalent to $\rho_{i} \leq \mathrm{min}_{R \in \mathrm{S}^{i}} \{ - \sum_{k \in R} \lambda_{k} + \mu_{i} + c_{R}^{i} \}$. Its right-hand side is equal to $- \mathrm{max}_{R \in \mathrm{S}^{i}}$  $\sum_{k \in R}  (\lambda_{k} -c_{ik}) + f_{i} + \mu_{i} = -\xi^{i}+f_{i}+\mu_{i}$. Also, the dual constraint corresponding to $y_{i}$ is $-\rho_{i} \leq 0$, and $\rho_{i}^{*}$ is feasible to this constraint. Hence, $(\lambda^{*}, \mu^{*}, \rho^{*}_{i})$ is a dual feasible solution. Because the dual objective function is independent of $\rho_{i}$,  $(\lambda^{*}, \mu^{*}, \rho_{i}^{*})$ is an optimal solution to the dual problem of \textbf{(MP)} with the coupling constraints. Thus, $\rho_{i}^{*}$ is the reduced cost of $y_{i}$. 
\end{proof}

\noindent As a result, variable fixing of $y_{i}=0$ is as follows. If $y_{i}$ is equal to zero with reduced cost $\rho_{i}^{*}$, and $\underline{Z} + \rho_{i}^{*} > \overline{Z}$, then $y_{i}$ can be fixed to zero. 

When $v_{i}=0$ in an optimal solution to \textbf{(MP)}, or equivaluently $y_{i}=1$, for some $i \in \mathrm{M}$ in an optimal solution to \textbf{(MP)}, we assume that the coupling constraints $-v_{i} - \sum_{R \in \mathrm{S}^{i}} z_{R}^{i} = -1$, $v_{i} \geq 0$ are included in \textbf{(MP)}. The following proposition shows that the reduced cost of $v_{i}$ can be calculated without adding the additional explicit coupling constraints to \textbf{(MP)}. 

\begin{proposition} \label{prop06}
Let $(\lambda^{*}, \mu^{*})$ be an optimal solution to \textbf{(DMP)}. For given $i \in \mathrm{M}$, let $\tau_{i}^{*} = \mu_i^{*}$. Let $\hat{\mu}^{*} \in \mathbb{R}^{m}$ be a vector where $\hat{\mu}^{*}_{k}=0$ if $k=i$, and $\hat{\mu}^{*}_{k}=\mu_{k}^{*}$ otherwise. Then, $(\lambda^{*}, \hat{\mu}^{*}, \tau_{i}^{*})$ is an optimal solution to the dual problem of \textbf{(MP)} with the coupling constraints $-v_{i} - \sum_{R \in \mathrm{S}^{i}} z_{R}^{i} = -1$, $v_{i} \geq 0$. Moreover, the reduced cost of $v_{i}$ is equal to $\tau_{i}^{*}$. 
\end{proposition}
\begin{proof}
After augmenting the coupling constraint to \textbf{(MP)}, constraints (\ref{eq30}) for $i \in \mathrm{M}$ are replaced by $-\tau_{i} + \sum_{k \in R} \lambda_{k} - \mu_{i} \leq c_{R}^{i}$ for $R \in \mathrm{S}^{i}$. These constraints are feasible for $\tau_{i} = \tau_{i}^{*} = \mu_i^{*}$ and $\mu_{i} = \hat{\mu}^{*}_{i}=0$. Also, the dual constraint corresponding to $v_{i}$ is $-\tau_{i} \leq 0$, and $\tau_{i}^{*}$ is feasible to this constraint. Hence, $(\lambda^{*}, \hat{\mu}^{*}, \tau_{i}^{*})$ is a dual feasible solution. Because the dual objective function is reconstructed to $\sum_{j \in \mathrm{N}} \lambda_{j} - \sum_{i \in \mathrm{M}} \mu_{i} - \tau_{i}$, its value is not changed when $(\lambda, \mu, \tau_{i}) = (\lambda^{*}, \hat{\mu}^{*}, \tau_{i}^{*})$. Therefore, $(\lambda^{*}, \hat{\mu}^{*}, \tau_{i}^{*})$ is an optimal solution to the dual problem of \textbf{(MP)} with the coupling constraints. Thus, $\tau_{i}^{*}$ is the reduced cost of $v_{i}$. 
\end{proof}

\noindent As a result, variable fixing of $v_{i}=0$, or equivalently $y_{i}=1$ is as follows. If $v_{i}$ is equal to zero with  reduced cost $\tau_{i}^{*}$, and $\underline{Z} + \tau_{i}^{*} > \overline{Z}$, then $v_{i}$ can be fixed to zero, or equivalently $y_{i}$ can be fixed to one. 

We note that fixing $x_{ij}$ to one may also be possible using similar approaches. However, the criterion for fixing seems to be more complicated to find, and we need to handle the situation that $x_{ij}=1$ and $y_{i}=0$ during the branch-and-price procedure. Therefore, we did not try fixing $x_{ij}$ to one in our study. 

%%%%%%%%%%%%%%%%%%%%%%%%%%%%%%%%%%%%%%%%%%%%%%%%%%%%%%%%
\subsection{Other implementation issues} \label{section35} 
%%%%%%%%%%%%%%%%%%%%%%%%%%%%%%%%%%%%%%%%%%%%%%%%%%%%%%%%

In the procedure of the branch-and-price algorithm, a depth-first search is applied for traversing the search tree. It is known to have relatively low performance compared to a best-first search, but it is useful to find feasible solutions and upper bounds earlier. Primal heuristics have been used widely in the branch-and-price algorithm to find good upper bounds. Although there have been many studies on the heuristics for the SSCFLP, little research has been reported on the heuristics for the robust SSCFLP under demand uncertainty. Therefore, we did not try using primal heuristics for our algorithm. We focused on verifying the effectiveness of the pure branch-and-price algorithm in our study. 

Infeasibility of \textbf{(RMP)} is also one of the implementation issues of the algorithm. When \textbf{(RMP)} becomes infeasible, the reason can be that \textbf{(MP)} is infeasible, or there are not enough columns to maintain the feasibility of \textbf{(RMP)}. However, it is hard to perceive the exact reason during the column generation procedure. Although there is Farkas pricing \citep{Desrosiers10} to detect whether a master problem is infeasible, it is as hard as optimizing a master problem. To avoid the infeasibility of \textbf{(RMP)} in advance, a dummy facility covering all customers with a very high fixed cost can be added. In our algorithm, We set the value of the fixed cost to two times the sum of all costs, i.e. $2(\sum_{i \in \mathrm{M}} \sum_{j \in \mathrm{N}}$ $c_{ij}$ $+\sum_{i \in \mathrm{M}} f_{i})$. The dummy facility has only one binary variable $z_{\mathrm{N}}^{0}$. If the value of $z_{\mathrm{N}}^{0}$ is nonzero after the algorithm solves \textbf{(RMP)} and there is no column having a negative reduced cost, then \textbf{(RMP)} is infeasible, and the algorithm can prune the node.

%%%%%%%%%%%%%%%%%%%%%%%%%%%%%%%%%%%%%%%%%%%%%%%%%%%%%%%%
\section{Computational experiments} \label{section4} 
%%%%%%%%%%%%%%%%%%%%%%%%%%%%%%%%%%%%%%%%%%%%%%%%%%%%%%%%

In this section, we report the performance  of the proposed branch-and-price algorithm. We implemented the algorithm using C++ with solvers of linear programming problems (CPLEX 12.9) and binary knapsack problems for the master problem and the subproblem, respectively. Computational results of solving the MIP model \textbf{(RP2)} using CPLEX 12.9 are also provided for comparison with our algorithm. Four different sets of the robust SSCFLP problem test instances are considered; The first two sets consist of benchmark instances used in  \citet{Delmaire99} and \citet{Holmberg99}, respectively with additional parameters for the maximum deviations, and the last two sets consist of randomly generated instances for detailed analysis. 

\begin{table}[t]
\caption{Problem size of SSCFLP instances.}
\centering \footnotesize 
\begin{tabular}{c|r|r|r|c}
\hline
set  & Instances    & $m$     & $n$      & S/D       \\ \hline
T1-1 & D1-D6 (6)    & 10    & 20     & 1.32-1.54 \\ 
T1-2 & D7-D17 (11)  & 15    & 30     & 1.33-3.15 \\ 
T1-3 & D18-D25 (8)  & 20    & 40     & 1.30-3.93 \\ 
T1-4 & D26-D33 (8)  & 20    & 50     & 1.27-4.06 \\ 
T1-5 & D34-D41 (8)  & 30    & 60     & 1.64-5.16 \\ 
T1-6 & D42-D49 (8)  & 30    & 70     & 1.43-3.01 \\ 
T1-7 & D50-D57 (8)  & 30    & 90     & 1.49-3.46 \\ \hline
T2-1 & H1-H12 (12)  & 10    & 50     & 1.37-2.06 \\ 
T2-2 & H13-H24 (12) & 20    & 50     & 2.77-3.50 \\ 
T2-3 & H25-H40 (16) & 30    & 150    & 3.03-6.06 \\ 
T2-4 & H41-H55 (15) & 10-30 & 70-100 & 1.52-8.28 \\ 
T2-5 & H56-H71 (16) & 30    & 200    & 1.97-3.95 \\ \hline
T3-1 & (10)         & 30    & 50     & 3.07-5.88 \\ 
T3-2 & (10)         & 30    & 70     & 3.23-5.93 \\ 
T3-3 & (10)         & 50    & 70     & 3.36-5.71 \\ 
T3-4 & (10)         & 50    & 100    & 2.36-4.43 \\ \hline
T4-1 & (10)         & 30    & 50     & 5.62-6.36 \\ 
T4-2 & (10)         & 30    & 70     & 3.92-4.49 \\ 
T4-3 & (10)         & 50    & 70     & 6.98-8.12 \\ 
T4-4 & (10)         & 50    & 100    & 4.45-5.26 \\ \hline
\end{tabular}
\label{table00}
\end{table}

%%%%%%%%%%%%%%%%%%%%%%%%%%%%%%%%%%%%%%%%%%%%%%%%%%%%%%%%
\subsection{Test instances} \label{section41} 
%%%%%%%%%%%%%%%%%%%%%%%%%%%%%%%%%%%%%%%%%%%%%%%%%%%%%%%%

We consider four different sets of test instances of the robust SSCFLP problem. Among them, the first two benchmark sets are directly taken from the previous literature, and the last two sets are generated for the additional experiments and the simulation experiments. The sizes of test instances are listed in Table \ref{table00}. In the table, we classify each test set into several subsets depending on the size of the problem. Also, S/D represents the ratio of the total capacity of all facilities over the total demand of customers. 

The first set (T1) of 57 test instances (D1-D57) ranging from ten candidate facility locations and 20 customers up to 30 locations and 90 customers were proposed by \citet{Delmaire99}. The second set (T2) of 71 test instances (H1-H71) ranging from ten candidate facility locations and 50 customers up to 30 locations and 200 customers were proposed by \citet{Holmberg99}. Test set (T1) and (T2) are divided into seven types and five types, respectively, depending on the instance size. Let $\mathrm{U}\{a,b\}$ be a random variable which has a discrete uniform distribution in $\{a,a+1,\cdots,b\}$. Each maximum deviation $b_{j}$ of customer demand is calculated by $b_{j} = \lfloor d_{j} \cdot \sigma_{j} \rfloor$, where $\sigma_{j}$ is taken from $\mathrm{U}\{100,500\} / 1000$. Each degree of robustness $\Gamma_{i}$ is fixed to 5.

The third set (T3) is generated based on the data generation scheme in \citet{Cornuejols91}. Nominal customer demands $d_j$, $j \in \mathrm{N}$ and capacities of facilities $s_i$, $i \in \mathrm{M}$ are firstly taken from $\mathrm{U}\{5,35\}$ and $\mathrm{U}\{10,160\}$, respectively, and the capacities are expanded by the same factor to adjust the ratio of the sum of capacities to the sum of demands appropriately. Set-up costs and allocation costs are obtained using  $f_i = \lfloor \mathrm{U}\{0,90\} + \mathrm{U}\{100,110\} \sqrt{s_i} \rfloor$ and $c_{ij} = \lfloor 10 d_{j} \cdot e_{ij} \rfloor$, respectively, where $e_{ij}$ is the Euclidean distance between facility $i$ and customer $j$ placed uniformly at random in a unit square. Maximum deviations $b_{j}$ of customer demands are decided as T1 and T2 and degree of robustness $\Gamma_{i}$ varies in $\{3,5,7\}$. 

The fourth set (T4) is generated based on the data generation scheme in \citet{Holmberg99}. Nominal customer demands $d_j$, capacities of facilities $s_i$, and set-up costs $f_i$ are taken from $\mathrm{U}\{10,50\}$, $\mathrm{U}\{100,500\}$, and $\mathrm{U}\{300,700\}$, respectively. Facilities and customers are placed uniformly at random in a square of size $190 \times 190$. Allocation cost $c_{ij}$ of allocating customer $j$ to facility $i$ is obtained by rounding down the Euclidean distance between them. Compared with instances of T2, based on the same reference \citep{Holmberg99}, set-up costs are relatively overvalued to reflect the realistic rates between allocation costs and set-up costs. Parameters $b_{j}$ and $\Gamma_{i}$ involved in the demand uncertainty and robustness are also obtained as T3.

The last two sets are designed so that we can examine the characteristics of instances for which our algorithm works well or not. For test set T3 and T4, we considered four different facility, customer pairs, i.e. $(30, 50)$, $(30, 70)$, $(50, 70)$, and $(50, 100)$. For each pair, we generated ten instances, totaling 40 instances for each test set.

%%%%%%%%%%%%%%%%%%%%%%%%%%%%%%%%%%%%%%%%%%%%%%%%%%%%%%%%
\subsection{Computational results} \label{section42} 
%%%%%%%%%%%%%%%%%%%%%%%%%%%%%%%%%%%%%%%%%%%%%%%%%%%%%%%%

\begin{table*}[t]
\caption{Computational results for instances of T1.}
\centering \tiny
\begin{tabular}{@{\extracolsep{4pt}}*{11}{r}@{}}
\hline
\multirow{2}{*}{set} & \multicolumn{6}{c}{Branch-and-price} & \multicolumn{4}{c}{CPLEX} \\ \cline{2-7} \cline{8-11}
& \#node & \#column & time(s) \textcolor{white}{(0)} & time-m(s) & time-s(s) & gapBP & \#node & time(s) \textcolor{white}{(0)} & gapLP & gapBC \\ \cline{1-7} \cline{8-11}
T1-1 & 30.0 & 216.5 & 0.10 \textcolor{white}{(0)} & 0.02 & 0.07 & 1.61 & 269482.8 & 273.12 (1) & 14.24 & 8.75 \\
T1-2 & 130.5 & 695.4 & 0.94 \textcolor{white}{(0)} & 0.15 & 0.68 & 0.76 & 387637.7 & 1000.40 (5) & 13.09 & 8.30 \\
T1-3 & 268.8 & 1822.9 & 6.47 \textcolor{white}{(0)} & 2.66 & 3.16 & 1.07* & 611844.6 & - (8) & 14.78* & 10.93* \\
T1-4 & 1044.0 & 6021.1 & 57.72 \textcolor{white}{(0)} & 34.69 & 13.39 & 0.81 & 364166.9 & - (8) & 14.13 & 10.04 \\
T1-5 & 21149.0 & 18093.9 & 1010.22 (1) & 589.47 & 216.19 & 2.23 & 420770.9 & - (8) & 19.96 & 17.27 \\
T1-6 & 22411.3 & 30811.5 & 46.77 (6) & 23.96 & 20.32 & 3.27 & 131969.0 & - (8) & 21.47 & 18.95 \\
T1-7 & 11734.0 & 43189.3 & 1984.88 (7) & 1606.20 & 134.44 & 5.37 & 166133.5 & - (8) & 22.51 & 20.32 \\ 
\hline 
\end{tabular}
  { \\ *Because one instance (D21) of T1-3 does not have a feasible solution, it was not included in computing the gaps.  \hfil \hfil \hfil \hfil \hfil }
\label{table01}
   \end{table*}

\begin{table*}[t]
\caption{Computational results for instances of T2.}
\centering \tiny
\begin{tabular}{@{\extracolsep{4pt}}*{11}{r}@{}}
\hline
\multirow{2}{*}{set}  & \multicolumn{6}{c}{Branch-and-price} & \multicolumn{4}{c}{CPLEX} \\ \cline{2-7} \cline{8-11}
& \#node & \#column & time(s) \textcolor{white}{(0)} & time-m(s) & time-s(s) & gapBP & \#node & time(s) \textcolor{white}{(0)} & gapLP & gapBC \\ \cline{1-7} \cline{8-11}
T2-1 & 7.2 & 868.2 & 0.24 \textcolor{white}{(0)} & 0.18 & 0.04 & 0.26 & 818.1 & 0.26 \textcolor{white}{(0)} & 2.47 & 1.62 \\ 
T2-2 & 35.3 & 1278.3 & 0.50 \textcolor{white}{(0)} & 0.36 & 0.07 & 1.19 & 1226.0 & 1.19 \textcolor{white}{(0)} & 3.32 & 1.89 \\
T2-3 & 66.8 & 16395.9 & 318.62 \textcolor{white}{(0)} & 281.68 & 7.45 & 0.81 & 9639.1 & 0.81 \textcolor{white}{(0)} & 1.26 & 0.95 \\
T2-4 & 11.8 & 4135.7 & 15.41 \textcolor{white}{(0)} & 13.75 & 0.52 & 0.28 & 284.4 & 0.28 \textcolor{white}{(0)} & 1.20 & 0.66 \\
T2-5 & 580.5 & 27370.0 & 733.37 (1) & 642.14 & 14.95 & 0.51 & 48461.7 & 70.72 (3) & 1.11 & 0.84 \\ \hline 
\end{tabular}
\label{table02}
\end{table*}

All computational experiments were performed on an Intel\textcircled{R} Core$^{\mathrm{TM}}$ i5-4670 CPU $@$ 3.40GHz PC with 24GB RAM. The branch-and-price algorithm was implemented with C++ programming language using Microsoft visual studio 2015, and it used ILOG CPLEX 12.9 for the LP solver of the algorithm. We also compared our result with the branch-and-cut algorithm solving \textbf{(RP2)} using ILOG CPLEX 12.9. 

In order to compare the experimental results of our branch-and-price algorithm and CPLEX, we report averaged test values for instances in each test set in Table \ref{table01}, Table \ref{table02}, Table \ref{table03}, and Table \ref{table04}. We report the number of nodes in the branch-and-bound tree (\#node), the number of generated columns (\#column), the overall computational time of the algorithm in seconds (time), the time for the master problem (time-m), and the subproblem (time-s), respectively. Also, we report the number of nodes (\#node) and the overall computational time in seconds (time) for CPLEX. We set the time limit to 3,600 seconds for both of the branch-and-price and CPLEX. If the algorithm could not find an optimal solution of an instance within the time limit, the instance was not included in computing the average computational time for (time), (time-m), and (time-s) of branch-and-price or (time) of CPLEX, and the number of unsolved instances is reported in the parentheses in the table. However, it was considered for obtaining the other numerical values.

We also compare the tightness of the LP-relaxation bound of \textbf{(AP)} and \textbf{(RP2)}. Let $Z_{best}^{AP}$ and $Z_{best}^{RP}$ be the optimal or best known objective function value of \textbf{(AP)} and \textbf{(RP2)}, respectively, and let $Z_{best}$ denote the best known objective value for the problem i.e. the smaller of $Z_{best}^{AP}$ and $Z_{best}^{RP}$. Also, Let $Z_{LP}^{AP}$, $Z_{LP}^{RP}$, and $Z_{LP}^{RPr}$ be the LP-relaxation bound of \textbf{(AP)}, \textbf{(RP2)}, and \textbf{(RP2)} with default cutting planes of CPLEX at the root node, respectively. We report the gap between the best known objective function value and the LP-relaxation bound of \textbf{(AP)} i.e. (gapBP) $=(Z_{best}-Z_{LP}^{AP}) / Z_{best} \times 100\%$ for each instance. We also report gaps between $Z_{best}$ and LP-relaxation bound of \textbf{(RP2)} without and with default cutting planes of CPLEX at the root node i.e. (gapLP) $=(Z_{best}-Z_{LP}^{RP}) / Z_{best} \times 100\% $ and (gapBC) $=(Z_{best}-Z_{LP}^{RPr})/ Z_{best} \times 100\%$ for comparison, respectively. The values are averaged and reported for each set of instances, and one problem in T1-3 which does not have a feasible solution was not included in calculating the average values of (gapBP), (gapLP), and (gapBC), respectively.

\begin{table*}[t]
\caption{Computational results for instances of T3.}
\centering \tiny
\begin{tabular}{@{\extracolsep{4pt}}*{12}{r}@{}}
\hline
\multirow{2}{*}{set} & \multirow{2}{*}{$\Gamma_i$} & \multicolumn{6}{c}{Branch-and-price} & \multicolumn{4}{c}{CPLEX} \\ \cline{3-8} \cline{9-12}
 & & \#node & \#column & time(s) \textcolor{white}{(0)} & time-m(s) & time-s(s) & gapBP & \#node & time(s) \textcolor{white}{(0)} & gapLP & gapBC \\ \cline{1-8} \cline{9-12}
\multirow{3}{*}{T3-1} & 3 & 214.4 & 6595.8 & 31.01 \textcolor{white}{(0)} & 19.40 & 7.35 & 1.58 & 63863.7 & 116.3 \textcolor{white}{(0)} & 3.95 & 2.55 \\
 & 5 & 413.6 & 7811.6 & 127.80 \textcolor{white}{(0)} & 106.92 & 11.52 & 1.67 & 509408.3 & 206.76 (3) & 4.80 & 3.24 \\
 & 7 & 418.6 & 8071.9 & 49.31 \textcolor{white}{(0)} & 32.56 & 9.96 & 1.44 & 557845.9 & 98.33 (5) & 4.52 & 3.04 \\ \cline{1-8} \cline{9-12}
\multirow{3}{*}{T3-2} & 3 & 1206.4 & 23909.1 & 209.93 (1) & 363.58 & 39.29 & 1.03 & 117207.7 & 350.06 (1) & 2.32 & 1.48 \\
 & 5 & 843.0 & 16092.6 & 500.25 \textcolor{white}{(0)} & 414.57 & 36.33 & 1.19 & 167151.9 & 72.84 (4) & 2.99 & 2.05 \\
 & 7 & 401.4 & 17985.1 & 188.54 \textcolor{white}{(0)} & 127.23 & 19.68 & 1.29 & 173455.9 & 320.27 (4) & 3.24 & 2.23 \\ \cline{1-8} \cline{9-12}
\multirow{3}{*}{T3-3} & 3 & 6653.4 & 35618.6 & 385.37 (3) & 241.03 & 61.60 & 1.26 & 176574.0 & 53.84 (8) & 4.81 & 3.11 \\
 & 5 & 5649.8 & 23719.5 & 715.73 (3) & 580.76 & 60.94 & 1.55 & 158088.8 & 108.71 (9) & 6.35 & 4.32 \\
 & 7 & 5550.0 & 27125.3 & 843.46 \textcolor{white}{(0)} & 521.74 & 149.35 & 1.02 & 155730.8 & 492.50 (9) & 5.55 & 3.54 \\ \cline{1-8} \cline{9-12}
\multirow{3}{*}{T3-4} & 3 & 3088.8 & 47337.4 & 1294.22 (2) & 854.31 & 163.84 & 1.18 & 162778 & 712.50 (6) & 3.16 & 2.11 \\ 
 & 5 & 2624.2 & 31374.9 & 1101.22 (3) & 908.24 & 103.23 & 1.19 & 143715.8 & 352.10 (7) & 4.01 & 2.58 \\ 
 & 7 & 2874.4 & 46713.1 & 1366.18 (2) & 930.44 & 133.69 & 1.19 & 143715.8 & 1248.93 (7) & 4.18 & 2.60 \\ \hline
\end{tabular}
\label{table03}
\end{table*}

\begin{table*}[t]
\caption{Computational results for instances of T4.}
\centering \tiny
\begin{tabular}{@{\extracolsep{4pt}}*{12}{r}@{}}
\hline
\multirow{2}{*}{set} & \multirow{2}{*}{$\Gamma_i$} & \multicolumn{6}{c}{Branch-and-price} & \multicolumn{4}{c}{CPLEX} \\ \cline{3-8} \cline{9-12}
 & & \#node & \#column & time(s) \textcolor{white}{(0)} & time-m(s) & time-s(s) & gapBP & \#node & time(s) \textcolor{white}{(0)} & gapLP & gapBC \\ \cline{1-8} \cline{9-12}
\multirow{3}{*}{T4-1} & 3 & 27.8 & 2945.7 & 2.84 \textcolor{white}{(0)} & 1.62 & 0.76 & 1.37 & 2310.5 & 2.42 \textcolor{white}{(0)} & 2.46 & 1.85 \\
 & 5 & 34.8 & 3005.2 & 7.80 \textcolor{white}{(0)} & 6.48 & 0.76 & 1.56 & 8060.4 & 8.34 \textcolor{white}{(0)} & 2.94 & 2.53 \\
 & 7 & 54.0 & 3510.2 & 6.69 \textcolor{white}{(0)} & 4.76 & 1.00 & 1.51 & 26153.3 & 29.86 \textcolor{white}{(0)} & 3.02 & 2.49 \\ \cline{1-8} \cline{9-12}
\multirow{3}{*}{T4-2} & 3 & 60.4 & 4830.5 & 11.14 \textcolor{white}{(0)} & 7.28 & 2.45 & 0.90 & 4049.5 & 4.98 \textcolor{white}{(0)} & 2.12 & 1.45 \\
 & 5 & 55.6 & 4829.4 & 22.60 \textcolor{white}{(0)} & 18.53 & 2.81 & 1.62 & 253384.9 & 254.15 (2) & 3.29 & 2.56 \\
 & 7 & 95.4 & 6264.7 & 23.38 \textcolor{white}{(0)} & 16.48 & 3.94 & 1.18 & 60902.4 & 112.41 \textcolor{white}{(0)} & 3.08 & 2.31 \\ \cline{1-8} \cline{9-12}
\multirow{3}{*}{T4-3} & 3 & 65.0 & 5914.3 & 16.07 \textcolor{white}{(0)} & 10.88 & 2.97 & 1.08 & 61910.9 & 353.44 \textcolor{white}{(0)} & 2.20 & 1.61 \\
 & 5 & 93.4 & 6428.2 & 48.87 \textcolor{white}{(0)} & 42.49 & 3.47 & 1.25 & 52953.9 & 417.51 \textcolor{white}{(0)} & 2.69 & 2.11 \\
 & 7 & 123.4 & 8560.8 & 36.93 \textcolor{white}{(0)} & 27.38 & 4.64 & 1.41 & 85825.4 & 360.41 (2) & 2.85 & 2.35 \\  \cline{1-8} \cline{9-12}
\multirow{3}{*}{T4-4} & 3 & 319.0 & 14698.4 & 109.30 \textcolor{white}{(0)} & 73.30 & 16.80 & 1.06 & 101862.5 & 127.78 (1) & 2.21 & 1.60 \\  
 & 5 & 196.8 & 11365.5 & 150.11 \textcolor{white}{(0)} & 127.71 & 12.44 & 0.85 & 83178.2 & 520.02 (2) & 2.51 & 1.77 \\  
 & 7 & 349.2 & 16484.4 & 133.92 \textcolor{white}{(0)} & 91.29 & 20.26 & 1.09 & 168073.0 & 998.60 (5) & 2.95 & 2.13 \\ \hline 
\end{tabular}
\label{table04}
\end{table*}

Table \ref{table01} and Table \ref{table02} present computational results of the branch-and-price algorithm and CPLEX for the benchmark instances of T1 and T2, respectively. Table 2 illustrates that our branch-and-price algorithm outperforms CPLEX by a wide margin for the instances of T1. The algorithm of CPLEX could obtain optimal solutions for only ten instances out of 57 instances, while our algorithm could find optimal solutions for 42 instances. Moreover, CPLEX could find an optimal solution faster than our algorithm for only one instance. Also, among the 15 instances unsolved within 3,600 seconds by our algorithm, CPLEX could find better solutions than our algorithm for only three instances. Overall, the branch-and-price algorithm is better than CPLEX for 53 instances out of 57 instances.  

However, the computational results for the instances of T2 in Table 3 show the opposite results, unlike the first ones. Our algorithm is better than CPLEX in terms of solving time for only 22 out of 71 instances, although our branch-and-price algorithm obtains optimal solutions for two of the three instances, which can not be exactly solved by CPLEX within 3,600 seconds. 

These conflicting results of computational experiments for two types of benchmark instances can be explained by the gap values between the upper and lower bounds for the MIP reformulations. For the instances of T1, the average value of the gap is 2.1\% for \textbf{(AP)} while it is 17.1\% for \textbf{(RP2)}. These fundamental differences appear in all experimental results of instances of T1, and CPLEX could not close this as much. CPLEX could decrease the average value of the gap to 13.4\% after adding the default cutting planes at the root node. However, for the instances of T2, the average value of the gap is 0.6\% for \textbf{(AP)} while it is 1.8\% for \textbf{(RP2)}, and the gap is decreased to 1.1\% on average after adding the default cutting planes of CPLEX. As a result, instances of T2 have relatively small gap values for \textbf{(RP2)}. Hence, they can be easily solved by CPLEX except for some large-sized instances.  
 
We also note that the number of nodes generated in the branch-and-bound tree was very small in our algorithm compared to CPLEX for all test instances. This may be due to the stronger bound provided by the LP-relaxation of \textbf{(AP)}, and such tendency may grow as we solve larger problems.

Table \ref{table03} and Table \ref{table04} present computational results of the branch-and-price algorithm and CPLEX for the randomly generated instances of T3 and T4, respectively. In total, there are 77 and 58 instances of T3 and T4, respectively, that our algorithm outperforms CPLEX in terms of the computational time and the quality of feasible solutions. For each gamma value, our branch and price algorithm outperforms CPLEX when $\Gamma_i=5$ (19 and 26 instances of T3 and T4, respectively) and $\Gamma_i=7$ (27 and 29 instances of T3 and T4, respectively), and our algorithm slightly underperforms CPLEX when $\Gamma_i=3$ (12 and 23 instances of T3 and T4, respectively). 

Moreover, our branch and price algorithm could solve 106 out of 120 instances of T3 and all 120 out of 120 instances of T4, while CPLEX could solve 57 instances of T3 and 108 instances of T4 within 3,600 seconds. In the case of the instances of T3, our branch and price algorithm could solve 34($\Gamma_i=3$), 34($\Gamma_i=5$), and 38($\Gamma_i=7)$ instances of T3, but CPLEX could solve 25($\Gamma_i=3$), 17($\Gamma_i=5$), and 15($\Gamma_i=7$). It shows that our algorithm maintains almost the same performance, but the performance of CPLEX decreases significantly when the value of gamma increases. In conclusion, our algorithm has solved the problems that CPLEX could not easily solve, although CPLEX showed better performance in terms of the computational time for some instances, mostly small-sized ones.

This difference in performance can be due to the gaps between the best known objective function value and the LP-relaxation bound. The gap for \textbf{(AP)} seems to be almost independent of the value of gamma, while the gap for \textbf{(RP2)} tends to grow proportional to the value of gamma. 

%%%%%%%%%%%%%%%%%%%%%%%%%%%%%%%%%%%%%%%%%%%%%%%%%%%%%%%% 
\section{Simulation experiments for evaluation of robust optimal solutions} \label{section5} 
%%%%%%%%%%%%%%%%%%%%%%%%%%%%%%%%%%%%%%%%%%%%%%%%%%%%%%%%

In this section, we report the results of simulation experiments to evaluate the robustness of the solutions of the robust SSCFLP. Trade-off between the robustness of the solutions and additional costs incurred is verified. It illustrates that the robust SSCFLP can deal with the demand uncertainty efficiently with minimal additional costs.  

%%%%%%%%%%%%%%%%%%%%%%%%%%%%%%%%%%%%%%%%%%%%%%%%%%%%%%%%
\subsection{Design of experiments} \label{section51} 
%%%%%%%%%%%%%%%%%%%%%%%%%%%%%%%%%%%%%%%%%%%%%%%%%%%%%%%%

For the simulation experiments, We generated two benchmark instances I3 and I4, which have 30 candidate facility locations and 70 customers, like as the instances of T3 and T4 in the previous section, respectively.

We solved each instance of the robust SSCFLP for all $\sigma_{j} \in \{0\%,$ $10\%,$ $20\%,$ $30\%,$ $40\%,$ $50\%\}$ and $\Gamma_{i} \in \{0,$ $1,$ $2,$ $3,$ $4,$ $5\}$. The two types of parameters, rate of the maximum possible variations $\sigma_{j}$ and degree of robustness $\Gamma_{i}$, control the level of robustness for the optimal solutions. When all $\sigma_{j}$ and $\Gamma_{i}$ are equal to zero, a solution of the original SSCFLP without demand uncertainty is obtained. We compared the robust solutions to the nominal solution in terms of penalty costs, additional available capacities, and the robustness of solutions. 

The robustness of solutions was measured by the empirical ratio of infeasibility using the Monte Carlo simulation. In the simulation, the demand of each customer $j \in \mathrm{N}$ is generated from the truncated normal distribution derived from the normal distribution with mean $d_j$ and standard deviation $d_j \cdot \Delta$, by cutting off the lower tail under $d_j \cdot (1 - 2  \Delta)$ in the normal distribution, where $\Delta$ is the level of variability in demands. The lower truncation of normal distribution prevents ridiculously small or negative value of demand. For each demand scenario, the feasibility of a scenario was confirmed by checking whether every opened facility could accommodate the demands of the assigned customers or not. The ratio of infeasibility was obtained by dividing the number of infeasible scenarios by 5,000 demand scenarios. We did the simulation with varying the level of variability in demands  $\Delta \in \{0, 0.05, \cdots, 0.40 \}$ for each robust solution.

\begin{figure}
\begin{subfigure}{.5\textwidth}
  \centering
  \begin{tikzpicture}[only marks, y=1cm, scale=0.75]
\begin{axis}[   width=6cm, height=6cm,  scale only axis, xmin=0, xmax=42,xlabel={$\Delta$},
                xtick={0,5,10,15,20,25,30,35,40},               
                xticklabels={0,,0.1,,0.2,,0.3,,0.4},
                xmajorgrids, ymin=0, ymax=100,
                ylabel={infeasibility (\%)},
                ytick={0,5,10,15,20,25,30,35,40,45,50,55,60,65,70,75,80,85,90,95,100},               
                yticklabels={0,,,,20,,,,40,,,,60,,,,80,,,,100},               
                ymajorgrids,axis lines*=left,line width=1.0pt,mark size=2.0pt,
                legend style={at={(1.1,1)},anchor=north west,draw=black,fill=white,align=left}
                ]
                
\addplot [color=black,solid,mark=*,mark options={solid},smooth]
    coordinates  {(0,0) (5,67.7) (10,76.9) (15,79.4) (20,80.9) (25,83.4) (30,84.8) (35,85.8) (40,87.0)};
    \addlegendentry{$\Gamma_{i}=0$};
\draw [latex-latex,semithick] (axis cs:5,67.7) -- node[pos=0.5,fill=white]{\tiny 63.6\%} (axis cs:5,4.1); 
\draw [latex-latex,semithick] (axis cs:10,76.9) -- node[pos=0.5,fill=white]{\tiny 37.9\%} (axis cs:10,39.0);
\draw [latex-latex,semithick] (axis cs:15,79.4) -- node[pos=0.5,fill=white]{\tiny 17.4\%} (axis cs:15,62.0);
\draw [latex-latex,semithick] (axis cs:20,80.9) -- node[pos=0.5,fill=white]{\tiny 10.2\%} (axis cs:20,70.8);

\addplot [color=black,dotted,mark=*,mark options={solid},smooth]
    coordinates {(0,0) (5,4.1) (10,39.0) (15,62.0) (20,70.8) (25,77.5) (30,80.4) (35,83.4) (40,85.4)};
    \addlegendentry{$\Gamma_{i}=1$};
\draw [latex-latex,semithick] (axis cs:10,39.0) -- node[pos=0.5,fill=white]{\tiny 32.2\%} (axis cs:10,6.8);      
\draw [latex-latex,semithick] (axis cs:15,62.0) -- node[pos=0.5,fill=white]{\tiny 32.1\%} (axis cs:15,29.9);
\draw [latex-latex,semithick] (axis cs:20,70.8) -- node[pos=0.5,fill=white]{\tiny 23.4\%} (axis cs:20,47.3);
\draw [latex-latex,semithick] (axis cs:25,77.5) -- node[pos=0.5,fill=white]{\tiny 15.4\%} (axis cs:25,62.2);
\draw [latex-latex,semithick] (axis cs:30,80.4) -- node[pos=0.5,fill=white]{\tiny 11.1\%} (axis cs:30,69.3);
               
\addplot [color=black,dash pattern=on 1pt off 3pt on 3pt off 3pt,mark=*,mark options={solid},smooth]
                coordinates {(0,0) (5,0) (10,6.8) (15,29.9) (20,47.3) (25,62.2) (30,69.3) (35,75.3) (40,80.4)};
                \addlegendentry{$\Gamma_{i}=2$};
\draw [latex-latex,semithick] (axis cs:15,29.9) -- node[pos=0.5,fill=white]{\tiny 24.7\%} (axis cs:15,5.2);
\draw [latex-latex,semithick] (axis cs:20,47.3) -- node[pos=0.5,fill=white]{\tiny 30.9\%} (axis cs:20,16.4);
\draw [latex-latex,semithick] (axis cs:25,62.2) -- node[pos=0.5,fill=white]{\tiny 32.0\%} (axis cs:25,30.1);
\draw [latex-latex,semithick] (axis cs:30,69.3) -- node[pos=0.5,fill=white]{\tiny 27.8\%} (axis cs:30,41.6); 
\draw [latex-latex,semithick] (axis cs:35,75.3) -- node[pos=0.5,fill=white]{\tiny 24.9\%} (axis cs:35,50.4); 
\draw [latex-latex,semithick] (axis cs:40,80.4) -- node[pos=0.5,fill=white]{\tiny 20.2\%} (axis cs:40,60.2);         
                
\addplot [color=black,dashed,mark=*,mark options={solid},smooth]
                coordinates {(0,0) (5,0) (10,0.2) (15,5.2) (20,16.4) (25,30.1) (30,41.6) (35,50.4) (40,60.2) };
                \addlegendentry{$\Gamma_{i}=3$};
\draw [latex-latex,semithick] (axis cs:20,16.4) -- node[pos=0.5,fill=white]{\tiny 13.5\%} (axis cs:20,3.0);
\draw [latex-latex,semithick] (axis cs:25,30.1) -- node[pos=0.5,fill=white]{\tiny 19.6\%} (axis cs:25,10.6);
\draw [latex-latex,semithick] (axis cs:30,41.6) -- node[pos=0.5,fill=white]{\tiny 20.1\%} (axis cs:30,21.4); 
\draw [latex-latex,semithick] (axis cs:35,50.4) -- node[pos=0.5,fill=white]{\tiny 19,2\%} (axis cs:35,31.3); 
\draw [latex-latex,semithick] (axis cs:40,60.2) -- node[pos=0.5,fill=white]{\tiny 16.9\%} (axis cs:40,43.3);                      
\addplot [color=black,dash pattern=on 3pt off 6pt on 6pt off 6pt,mark=*,mark options={solid},smooth]
                coordinates {(0,0) (5,0) (10,0) (15,0.3) (20,3.0) (25,10.6) (30,21.4) (35,31.3) (40,43.3)};
                \addlegendentry{$\Gamma_{i}=4$};
\addplot [color=black,dotted,mark=*,mark options={solid},smooth]
                coordinates {(0,0) (5,0) (10,0) (15,0.2) (20,2.3) (25,9.8) (30,19.3) (35,29.1) (40,42.1)};
                \addlegendentry{$\Gamma_{i}=5$};                                                                     
\end{axis}
\end{tikzpicture}  
  \caption{I3}
  \label{fig01}
\end{subfigure} 
\begin{subfigure}{.5\textwidth}
  \centering
  \begin{tikzpicture}[only marks, y=1cm, scale=0.75]
\begin{axis}[   width=6cm, height=6cm,  scale only axis, xmin=0, xmax=42,xlabel={$\Delta$},
                xtick={0,5,10,15,20,25,30,35,40},               
                xticklabels={0,,0.1,,0.2,,0.3,,0.4},
                xmajorgrids, ymin=0, ymax=100,
                ylabel={infeasibility (\%)},
                ytick={0,5,10,15,20,25,30,35,40,45,50,55,60,65,70,75,80,85,90,95,100},               
                yticklabels={0,,,,20,,,,40,,,,60,,,,80,,,,100},               
                ymajorgrids,axis lines*=left,line width=1.0pt,mark size=2.0pt,
                legend style={at={(1.1,1)},anchor=north west,draw=black,fill=white,align=left}
                ]
                
\addplot [color=black,solid,mark=*,mark options={solid},smooth]
    coordinates  {(0,0) (5,78.4) (10,92.3) (15,95.7) (20,96.3) (25,97.0) (30,97.1) (35,97.3) (40,97.5)};
    \addlegendentry{$\Gamma_{i}=0$};
\draw [latex-latex,semithick] (axis cs:5,78.4) -- node[pos=0.5,fill=white]{\tiny 76.8\%} (axis cs:5,1.7);
\draw [latex-latex,semithick] (axis cs:10,92.3) -- node[pos=0.5,fill=white]{\tiny 56.2\%} (axis cs:10,36.1);
\draw [latex-latex,semithick] (axis cs:15,95.7) -- node[pos=0.5,fill=white]{\tiny 30.6\%} (axis cs:15,65.1);
\draw [latex-latex,semithick] (axis cs:20,96.3) -- node[pos=0.5,fill=white]{\tiny 18.6\%} (axis cs:20,77.7);
\draw [latex-latex,semithick] (axis cs:25,97.0) -- node[pos=0.5,fill=white]{\tiny 12.4\%} (axis cs:25,84.6);

\addplot [color=black,dotted,mark=*,mark options={solid},smooth]
    coordinates {(0,0) (5,1.7) (10,36.1) (15,65.1) (20,77.7) (25,84.6) (30,88.5) (35,90.6) (40,92.3)};
    \addlegendentry{$\Gamma_{i}=1$};
\draw [latex-latex,semithick] (axis cs:10,36.1) -- node[pos=0.5,fill=white]{\tiny 35.0\%} (axis cs:10,1.1);
\draw [latex-latex,semithick] (axis cs:15,65.1) -- node[pos=0.5,fill=white]{\tiny 51.6\%} (axis cs:15,13.5);
\draw [latex-latex,semithick] (axis cs:20,77.7) -- node[pos=0.5,fill=white]{\tiny 47.1\%} (axis cs:20,30.6);
\draw [latex-latex,semithick] (axis cs:25,84.6) -- node[pos=0.5,fill=white]{\tiny 36.3\%} (axis cs:25,48.2);
\draw [latex-latex,semithick] (axis cs:30,88.5) -- node[pos=0.5,fill=white]{\tiny 26.9\%} (axis cs:30,61.6);
\draw [latex-latex,semithick] (axis cs:35,90.6) -- node[pos=0.5,fill=white]{\tiny 21.1\%} (axis cs:35,69.5);
\draw [latex-latex,semithick] (axis cs:40,92.3) -- node[pos=0.5,fill=white]{\tiny 16.8\%} (axis cs:40,75.5);
               
\addplot [color=black,dash pattern=on 1pt off 3pt on 3pt off 3pt,mark=*,mark options={solid},smooth]
                coordinates {(0,0) (5,0) (10,1.1) (15,13.5) (20,30.6) (25,48.2) (30,61.6) (35,69.5) (40,75.5)};
                \addlegendentry{$\Gamma_{i}=2$}; 
\draw [latex-latex,semithick] (axis cs:20,30.6) -- node[pos=0.5,fill=white]{\tiny 15.1\%} (axis cs:20,15.5);
\draw [latex-latex,semithick] (axis cs:25,48.2) -- node[pos=0.5,fill=white]{\tiny 16.1\%} (axis cs:25,32.2);
\draw [latex-latex,semithick] (axis cs:30,61.6) -- node[pos=0.5,fill=white]{\tiny 13.9\%} (axis cs:30,47.7);
\draw [latex-latex,semithick] (axis cs:35,69.5) -- node[pos=0.5,fill=white]{\tiny 12.0\%} (axis cs:35,57.5);      
                
\addplot [color=black,dashed,mark=*,mark options={solid},smooth]
                coordinates {(0,0) (5,0) (10,0.1) (15,4.1) (20,15.5) (25,32.2) (30,47.7) (35,57.5) (40,67.5) };
                \addlegendentry{$\Gamma_{i}=3$};   
\draw [latex-latex,semithick] (axis cs:20,15.5) -- node[pos=0.5,fill=white]{\tiny 11.1\%} (axis cs:20,4.4);
\draw [latex-latex,semithick] (axis cs:25,32.2) -- node[pos=0.5,fill=white]{\tiny 17.7\%} (axis cs:25,14.5);
\draw [latex-latex,semithick] (axis cs:30,47.7) -- node[pos=0.5,fill=white]{\tiny 19.8\%} (axis cs:30,27.9);
\draw [latex-latex,semithick] (axis cs:35,57.5) -- node[pos=0.5,fill=white]{\tiny 17.4\%} (axis cs:35,40.1);  
\draw [latex-latex,semithick] (axis cs:40,67.5) -- node[pos=0.5,fill=white]{\tiny 16.4\%} (axis cs:40,51.0);                     
\addplot [color=black,dash pattern=on 3pt off 6pt on 6pt off 6pt,mark=*,mark options={solid},smooth]
                coordinates {(0,0) (5,0) (10,0) (15,0.7) (20,4.4) (25,14.5) (30,27.9) (35,40.1) (40,51.0)};
                \addlegendentry{$\Gamma_{i}=4$};
\addplot [color=black,dotted,mark=*,mark options={solid},smooth]
                coordinates {(0,0) (5,0) (10,0) (15,0.1) (20,1.6) (25,7.9) (30,20.0) (35,31.6) (40,45.0)};
                \addlegendentry{$\Gamma_{i}=5$};                                                                     
\end{axis}
\end{tikzpicture}  
  
  \caption{I4}
  \label{fig02}
\end{subfigure}

\caption{The percentage ratio of infeasible scenarios when $\sigma_{j}=30\%$.}
\label{fig0102}
\end{figure}
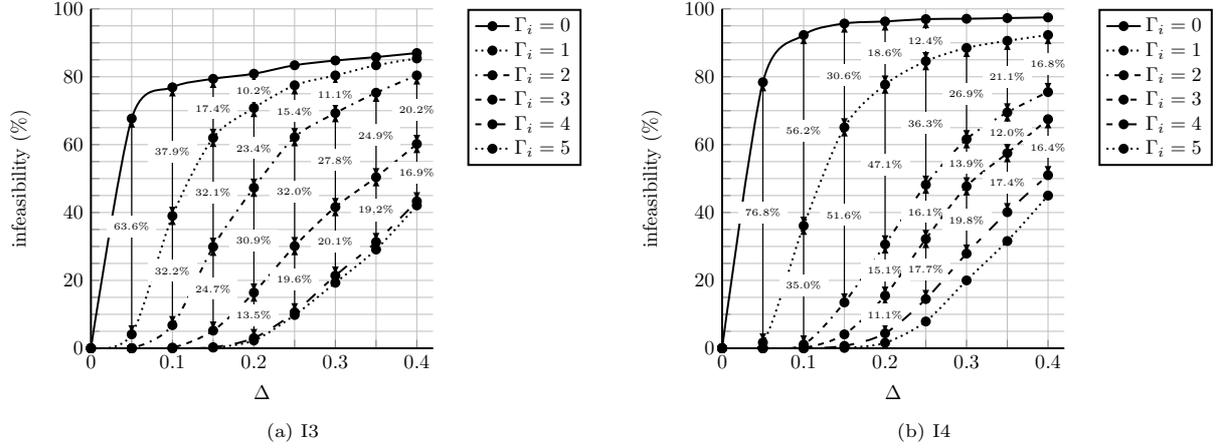

%%%%%%%%%%%%%%%%%%%%%%%%%%%%%%%%%%%%%%%%%%%%%%%%%%%%%%%%
\subsection{Results of experiments and analysis} \label{section52} 
%%%%%%%%%%%%%%%%%%%%%%%%%%%%%%%%%%%%%%%%%%%%%%%%%%%%%%%%

Figure \ref{fig01} and Figure \ref{fig02} illustrate the percentage ratio of infeasibility under the same rate of maximum possible variation $\sigma_j = 30\%$ for I3 and I4, respectively. It shows that the robustness of solutions depends on the degree of robustness $\Gamma_i=0,1,2,3,4,5$ and the rate of demand variation scenario $\Delta=0, 0.05, \cdots, 0.40$. The ratio of infeasibility becomes smaller as the degree of robustness increases. For example, when the rate of demand variation $\Delta$ is equal to $0.05$, the ratio of infeasibility can be improved by $63.6\%$ on I3 and $76.8\%$ on I4 by increasing the value of $\Gamma_i$ from zero to one, respectively. For the better result, we need to pay $0.5\%$ and $1.8\%$ additional penalty costs for I3 and I4 compared to the nominal solution without demand uncertainty, respectively. We also report that the additional penalty costs are $5.3\%$, $6.4\%$, $7.5\%$, and $8.5\%$ for I3 and $3.7\%$, $5.0\%$, $6.6\%$, and $8.0\%$ for I4 when $\Gamma_{i}=2,3,4,5$, respectively. Actually, when the rate of demand variation $\Delta$ increases, the degree of robustness $\Gamma_{i}$ should be higher for improving the rate of feasibility, and it increases the penalty costs. However, we can observe that the penalty costs are not large, compared with the improvement on the robustness of the solutions. 

\begin{figure}
\begin{subfigure}{.5\textwidth}
  \centering
  \begin{tikzpicture}[only marks, y=1cm, scale=0.75]
\begin{axis}[   width=6cm, height=6cm,  scale only axis, xmin=0, xmax=42,xlabel={$\Delta$},
                xtick={0,5,10,15,20,25,30,35,40},               
                xticklabels={0,,0.1,,0.2,,0.3,,0.4},
                xmajorgrids, ymin=0, ymax=100,
                ylabel={infeasibility (\%)},
                ytick={0,5,10,15,20,25,30,35,40,45,50,55,60,65,70,75,80,85,90,95,100},               
                yticklabels={0,,,,20,,,,40,,,,60,,,,80,,,,100},               
                ymajorgrids,axis lines*=left,line width=1.0pt,mark size=2.0pt,
                legend style={at={(1.1,1)},anchor=north west,draw=black,fill=white,align=left}
                ]
                
\addplot [color=black,solid,mark=*,mark options={solid},smooth]
    coordinates  {(0,0) (5,67.7) (10,76.9) (15,79.4) (20,80.9) (25,83.4) (30,84.8) (35,85.8) (40,87.0)};
    \addlegendentry{$\sigma_{j}=0\%$};
\draw [latex-latex,semithick] (axis cs:5,67.7) -- node[pos=0.5,fill=white]{\tiny 62.8\%} (axis cs:5,4.9);
\draw [latex-latex,semithick] (axis cs:10,76.9) -- node[pos=0.5,fill=white]{\tiny 37.5\%} (axis cs:10,39.4);
\draw [latex-latex,semithick] (axis cs:15,79.4) -- node[pos=0.5,fill=white]{\tiny 17.3\%} (axis cs:15,62.1);

\addplot [color=black,dotted,mark=*,mark options={solid},smooth]
    coordinates {(0,0) (5,4.9) (10,39.4) (15,62.1) (20,71.2) (25,77.2) (30,80.8) (35,83.2) (40,85.3)};
    \addlegendentry{$\sigma_{j}=10\%$};
\draw [latex-latex,semithick] (axis cs:10,39.4) -- node[pos=0.5,fill=white]{\tiny 32.7\%} (axis cs:10,6.8);
\draw [latex-latex,semithick] (axis cs:15,62.1) -- node[pos=0.5,fill=white]{\tiny 30.7\%} (axis cs:15,31.4);
\draw [latex-latex,semithick] (axis cs:20,71.2) -- node[pos=0.5,fill=white]{\tiny 21.9\%} (axis cs:20,49.4);
\draw [latex-latex,semithick] (axis cs:25,77.2) -- node[pos=0.5,fill=white]{\tiny 12.0\%} (axis cs:25,65.3);
               
\addplot [color=black,dash pattern=on 1pt off 3pt on 3pt off 3pt,mark=*,mark options={solid},smooth]
                coordinates {(0,0) (5,0) (10,6.8) (15,31.4) (20,49.4) (25,65.3) (30,72.9) (35,77.2) (40,81.3)};
                \addlegendentry{$\sigma_{j}=20\%$};    
\draw [latex-latex,semithick] (axis cs:15,31.4) -- node[pos=0.5,fill=white]{\tiny 26.2\%} (axis cs:15,5.2);
\draw [latex-latex,semithick] (axis cs:20,49.4) -- node[pos=0.5,fill=white]{\tiny 32.9\%} (axis cs:20,16.4);
\draw [latex-latex,semithick] (axis cs:25,65.3) -- node[pos=0.5,fill=white]{\tiny 35.1\%} (axis cs:25,30.1);
\draw [latex-latex,semithick] (axis cs:30,72.9) -- node[pos=0.5,fill=white]{\tiny 31.3\%} (axis cs:30,41.6);
\draw [latex-latex,semithick] (axis cs:35,77.2) -- node[pos=0.5,fill=white]{\tiny 26.8\%} (axis cs:35,50.4);
\draw [latex-latex,semithick] (axis cs:40,81.3) -- node[pos=0.5,fill=white]{\tiny 21.1\%} (axis cs:40,60.2);
                
\addplot [color=black,dashed,mark=*,mark options={solid},smooth]
                coordinates {(0,0) (5,0) (10,0.2) (15,5.2) (20,16.4) (25,30.1) (30,41.6) (35,50.4) (40,60.2) };
                \addlegendentry{$\sigma_{j}=30\%$};                     
\draw [latex-latex,semithick] (axis cs:20,16.4) -- node[pos=0.5,fill=white]{\tiny 12.3\%} (axis cs:20,4.2);
\draw [latex-latex,semithick] (axis cs:25,30.1) -- node[pos=0.5,fill=white]{\tiny 17.5\%} (axis cs:25,12.6);
\draw [latex-latex,semithick] (axis cs:30,41.6) -- node[pos=0.5,fill=white]{\tiny 19.1\%} (axis cs:30,22.5);
\draw [latex-latex,semithick] (axis cs:35,50.4) -- node[pos=0.5,fill=white]{\tiny 18.2\%} (axis cs:35,32.2);
\draw [latex-latex,semithick] (axis cs:40,60.2) -- node[pos=0.5,fill=white]{\tiny 16.4\%} (axis cs:40,43.7);
\addplot [color=black,dash pattern=on 3pt off 6pt on 6pt off 6pt,mark=*,mark options={solid},smooth]
                coordinates {(0,0) (5,0) (10,0) (15,0.6) (20,4.2) (25,12.6) (30,22.5) (35,32.2) (40,43.7)};
                \addlegendentry{$\sigma_{j}=40\%$};
\draw [latex-latex,semithick] (axis cs:30,22.5) -- node[pos=0.5,fill=white]{\tiny 13.8\%} (axis cs:30,8.7);
\draw [latex-latex,semithick] (axis cs:35,32.2) -- node[pos=0.5,fill=white]{\tiny 15.7\%} (axis cs:35,16.6);
\draw [latex-latex,semithick] (axis cs:40,43.7) -- node[pos=0.5,fill=white]{\tiny 17.9\%} (axis cs:40,25.8);
\addplot [color=black,dotted,mark=*,mark options={solid},smooth]
                coordinates {(0,0) (5,0) (10,0) (15,0) (20,0.5) (25,3.5) (30,8.7) (35,16.6) (40,25.8)};
                \addlegendentry{$\sigma_{j}=50\%$};                                                                     
\end{axis}
\end{tikzpicture}  
  \caption{I3}
  \label{fig03}
\end{subfigure}
\begin{subfigure}{.5\textwidth}
  \centering
\begin{tikzpicture}[only marks, y=1cm, scale=0.75]
\begin{axis}[   width=6cm, height=6cm,  scale only axis, xmin=0, xmax=42,xlabel={$\Delta$},
                xtick={0,5,10,15,20,25,30,35,40},               
                xticklabels={0,,0.1,,0.2,,0.3,,0.4},
                xmajorgrids, ymin=0, ymax=100,
                ylabel={infeasibility (\%)},
                ytick={0,5,10,15,20,25,30,35,40,45,50,55,60,65,70,75,80,85,90,95,100},               
                yticklabels={0,,,,20,,,,40,,,,60,,,,80,,,,100},               
                ymajorgrids,axis lines*=left,line width=1.0pt,mark size=2.0pt,
                legend style={at={(1.1,1)},anchor=north west,draw=black,fill=white,align=left}
                ]
                
\addplot [color=black,solid,mark=*,mark options={solid},smooth]
    coordinates  {(0,0) (5,78.4) (10,92.3) (15,95.7) (20,96.3) (25,97.0) (30,97.1) (35,97.3) (40,97.5)};
    \addlegendentry{$\sigma_{j}=0\%$};
\draw [latex-latex,semithick] (axis cs:5,78.4) -- node[pos=0.5,fill=white]{\tiny 73.8\%} (axis cs:5,4.7);
\draw [latex-latex,semithick] (axis cs:10,92.3) -- node[pos=0.5,fill=white]{\tiny 49.8\%} (axis cs:10,42.6);
\draw [latex-latex,semithick] (axis cs:15,95.7) -- node[pos=0.5,fill=white]{\tiny 25.0\%} (axis cs:15,70.7);
\draw [latex-latex,semithick] (axis cs:20,96.3) -- node[pos=0.5,fill=white]{\tiny 15.0\%} (axis cs:20,81.3);
    
\addplot [color=black,dotted,mark=*,mark options={solid},smooth]
    coordinates {(0,0) (5,4.7) (10,42.6) (15,70.7) (20,81.3) (25,87.9) (30,90.5) (35,92.3) (40,93.1)};
    \addlegendentry{$\sigma_{j}=10\%$};
\draw [latex-latex,semithick] (axis cs:10,42.6) -- node[pos=0.5,fill=white]{\tiny 41.0\%} (axis cs:10,1.5);
\draw [latex-latex,semithick] (axis cs:15,70.7) -- node[pos=0.5,fill=white]{\tiny 52.2\%} (axis cs:15,18.5);
\draw [latex-latex,semithick] (axis cs:20,81.3) -- node[pos=0.5,fill=white]{\tiny 43.5\%} (axis cs:20,37.8);
\draw [latex-latex,semithick] (axis cs:25,87.9) -- node[pos=0.5,fill=white]{\tiny 31.8\%} (axis cs:25,56.1);
\draw [latex-latex,semithick] (axis cs:30,90.5) -- node[pos=0.5,fill=white]{\tiny 22.8\%} (axis cs:30,67.7);
\draw [latex-latex,semithick] (axis cs:35,92.3) -- node[pos=0.5,fill=white]{\tiny 17.5\%} (axis cs:35,74.8);
\draw [latex-latex,semithick] (axis cs:40,93.1) -- node[pos=0.5,fill=white]{\tiny 12.9\%} (axis cs:40,80.2);
               
\addplot [color=black,dash pattern=on 1pt off 3pt on 3pt off 3pt,mark=*,mark options={solid},smooth]
                coordinates {(0,0) (5,0) (10,1.5) (15,18.5) (20,37.8) (25,56.1) (30,67.7) (35,74.8) (40,80.2)};
                \addlegendentry{$\sigma_{j}=20\%$};    
\draw [latex-latex,semithick] (axis cs:15,18.5) -- node[pos=0.5,fill=white]{\tiny 14.3\%} (axis cs:15,4.1);
\draw [latex-latex,semithick] (axis cs:20,37.8) -- node[pos=0.5,fill=white]{\tiny 22.3\%} (axis cs:20,15.5);
\draw [latex-latex,semithick] (axis cs:25,56.1) -- node[pos=0.5,fill=white]{\tiny 23.9\%} (axis cs:25,32.2);
\draw [latex-latex,semithick] (axis cs:30,67.7) -- node[pos=0.5,fill=white]{\tiny 20.0\%} (axis cs:30,47.7);
\draw [latex-latex,semithick] (axis cs:35,74.8) -- node[pos=0.5,fill=white]{\tiny 17.3\%} (axis cs:35,57.5);
\draw [latex-latex,semithick] (axis cs:40,80.2) -- node[pos=0.5,fill=white]{\tiny 12.7\%} (axis cs:40,67.5);
                
\addplot [color=black,dashed,mark=*,mark options={solid},smooth]
                coordinates {(0,0) (5,0) (10,0.1) (15,4.1) (20,15.5) (25,32.2) (30,47.7) (35,57.5) (40,67.5) };
                \addlegendentry{$\sigma_{j}=30\%$};   
\draw [latex-latex,semithick] (axis cs:20,15.5) -- node[pos=0.5,fill=white]{\tiny 12.0\%} (axis cs:20,3.5);
\draw [latex-latex,semithick] (axis cs:25,32.2) -- node[pos=0.5,fill=white]{\tiny 19.4\%} (axis cs:25,12.7);
\draw [latex-latex,semithick] (axis cs:30,47.7) -- node[pos=0.5,fill=white]{\tiny 20.8\%} (axis cs:30,26.9);
\draw [latex-latex,semithick] (axis cs:35,57.5) -- node[pos=0.5,fill=white]{\tiny 17.9\%} (axis cs:35,39.6);
\draw [latex-latex,semithick] (axis cs:40,67.5) -- node[pos=0.5,fill=white]{\tiny 15.7\%} (axis cs:40,51.8);        
                
\addplot [color=black,dash pattern=on 3pt off 6pt on 6pt off 6pt,mark=*,mark options={solid},smooth]
                coordinates {(0,0) (5,0) (10,0) (15,0.3) (20,3.5) (25,12.7) (30,26.9) (35,39.6) (40,51.8)};
                \addlegendentry{$\sigma_{j}=40\%$};
\draw [latex-latex,semithick] (axis cs:25,12.7) -- node[pos=0.5,fill=white]{\tiny 11.0\%} (axis cs:25,1.8);
\draw [latex-latex,semithick] (axis cs:30,26.9) -- node[pos=0.5,fill=white]{\tiny 19.4\%} (axis cs:30,7.5);
\draw [latex-latex,semithick] (axis cs:35,39.6) -- node[pos=0.5,fill=white]{\tiny 23.9\%} (axis cs:35,15.7);
\draw [latex-latex,semithick] (axis cs:40,51.8) -- node[pos=0.5,fill=white]{\tiny 25.1\%} (axis cs:40,26.7);  
                
\addplot [color=black,dotted,mark=*,mark options={solid},smooth]
                coordinates {(0,0) (5,0) (10,0) (15,0) (20,0.2) (25,1.8) (30,7.5) (35,15.7) (40,26.7)};
                \addlegendentry{$\sigma_{j}=50\%$};                                                                     
\end{axis}
\end{tikzpicture}  
  \caption{I4}
  \label{fig04}
\end{subfigure}
\caption{The percentage ratio of infeasible scenarios when $\Gamma_{i}=3$.}
\label{fig0304}
\end{figure}
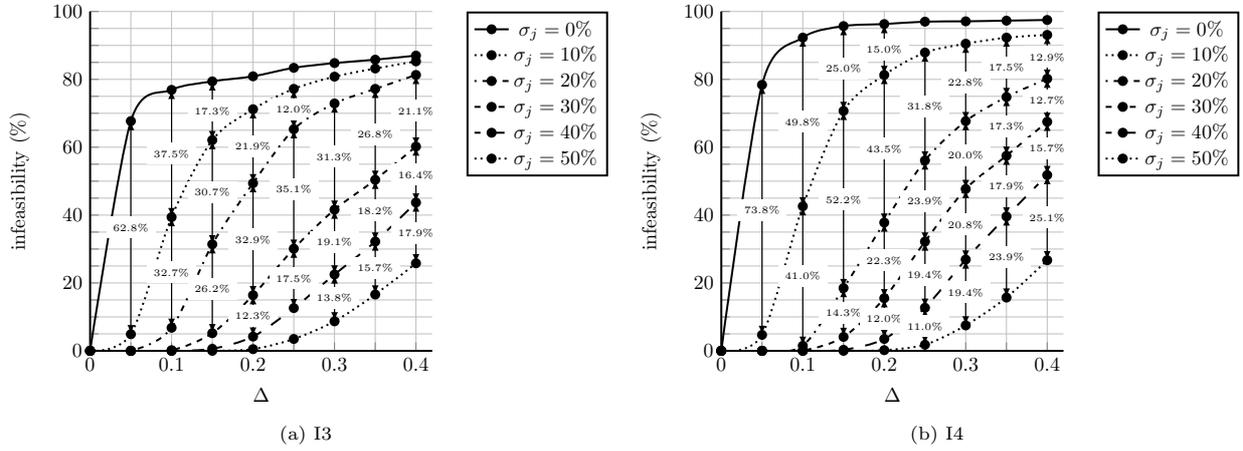

Figure \ref{fig03} and Figure \ref{fig04} illustrate the ratio of infeasibility under the same degree of robustness $\Gamma_{i} = 3$ for I3 and I4, respectively. They show that the ratio of infeasibility  depends on the rate of maximum possible variation $\sigma_j=0\%,10\%,20\%,30\%,40\%,50\%$ and the rate of demand variation scenario $\Delta=0, 0.05,\cdots,0.40$. The ratio of infeasibility becomes better as the rate of maximum possible variation increases. For example, when the rate of demand variation $\Delta$ is equal to $0.05$, the ratio of infeasibility can be improved by $62.8\%$ on I3 and $73.8\%$ on I4 when $\sigma_j$ is increased from zero to ten percents, respectively. It means that the robust solution obtained when $\sigma_j = 10\%$ is much better protected against infeasibility compared to the solution with $\sigma_j = 0\%$ (i.e. nominal problem). For the better result, we pay $0.4\%$ and $1.6\%$ additional penalty costs for I3 and I4 compared to the nominal solution without demand uncertainty, respectively. We also report that the additional penalty costs are $5.1\%$, $6.4\%$, $7.6\%$, and $8.8\%$ for I3 and $3.5\%$, $5.0\%$, $7.1\%$, and $9.2\%$ for I4, when $\sigma_j=20\%,30\%,40\%,50\%$, respectively. A similar phenomenon can be observed when the rate of maximum possible variation $\sigma_j$ is changed, compared with varying the degree of robustness $\Gamma_{i}$.

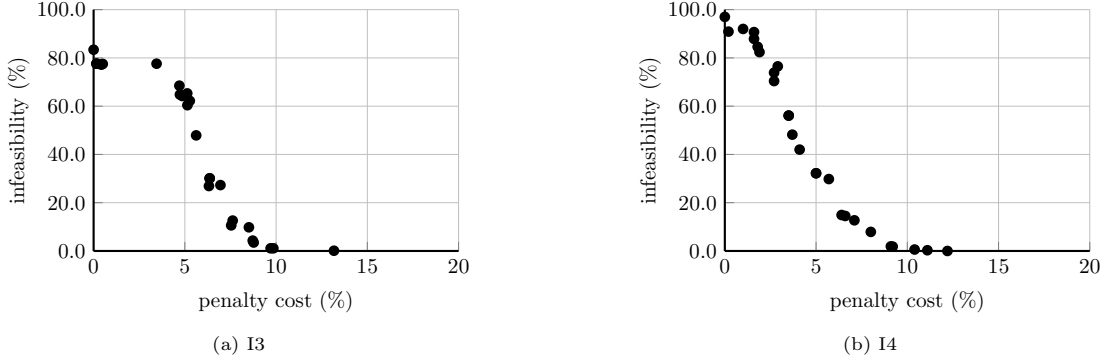
\begin{figure*}
\begin{subfigure}{.5\textwidth}
  \centering
\begin{tikzpicture}[only marks, y=1cm, scale=0.8]
\begin{axis}[   width=6cm,
                height=4cm,
                scale only axis,
                xmin=0, xmax=20, xlabel={penalty cost (\%)},
                xtick={0,5,10,15,20}, xticklabels={0,5,10,15,20},
                xmajorgrids,
                ymin=0, ymax=100, ylabel={infeasibility (\%)},
                ytick={0,20,40,60,80,100},
                yticklabels={0.0,20.0,40.0,60.0,80.0,100.0},
                ymajorgrids,
                axis lines*=left,
                line width=1.0pt,
                mark size=2.0pt,
                legend style={at={(0.9,0.9)},anchor=north east,draw=black,fill=white,align=left}
                ]
\addplot [color=black,mark=*]
                coordinates  {(0,83.4) (0.16,77.6) (0.16,77.6) (0.42,77.2) (3.45,77.6) (4.70,68.5) (0.16,77.6) (4.73,64.8) (5.13,65.3) (5.62,47.9) (6.32,26.9)  (0.50,77.5) (5.27,62.2) (6.36,30.1) (7.54,10.6) (8.51,9.8) (4.86,64.2) (6.36,30.1) (7.62,12.6) (8.72,4.3) (9.70,1.1) (5.14,60.4) (6.95,27.3) (8.77,3.5) (9.85,1.1) (13.17,0.1) };

\end{axis}              
\end{tikzpicture}    
  \caption{I3}
  \label{fig05}
\end{subfigure}
\begin{subfigure}{.5\textwidth}
  \centering
\begin{tikzpicture}[only marks, y=1cm, scale=0.8]
\begin{axis}[   width=6cm,
                height=4cm,
                scale only axis,
                xmin=0, xmax=20, xlabel={penalty cost (\%)},
                xtick={0,5,10,15,20}, xticklabels={0,5,10,15,20},
                xmajorgrids,
                ymin=0, ymax=100, ylabel={infeasibility (\%)},
                ytick={0,20,40,60,80,100},
                yticklabels={0.0,20.0,40.0,60.0,80.0,100.0},
                ymajorgrids,
                axis lines*=left,
                line width=1.0pt,
                mark size=2.0pt,
                legend style={at={(0.9,0.9)},anchor=north east,draw=black,fill=white,align=left}
                ]
\addplot [color=black,mark=*]
                coordinates  {(0,97.0) (0.2,90.9) (1.0,92.0) (1.6,87.9) (1.9,82.4) (2.7,73.9)  (1.6,90.7) (2.7,70.4) (3.5,56.1) (4.1,42.0) (5.7,29.8)  (1.8,84.6) (3.7,48.2) (5.0,32.2) (6.6,14.5) (8.0,7.9)  (2.9,76.5) (5.0,32.2) (7.1,12.7) (9.1,1.9) (10.4,0.6) (3.5,56.1) (6.4,14.9) (9.2,1.8) (11.1,0.3) (12.2,0.0)  };
\end{axis}              
\end{tikzpicture}  
  \caption{I4}
  \label{fig06}
\end{subfigure}
\caption{Relationship between the ratio of infeasibility and the penalty cost, $\sigma_{j} \times \Gamma_{i}$ $\in$ $\{10\%,$ $20\%,$ $30\%,$ $40\%,$ $50\%\}$ $\times$ $\{1,$ $2,$ $3,$ $4,$ $5\}$.}
\label{fig0506}
\end{figure*}

Figure \ref{fig05} and Figure \ref{fig06} illustrate relationship between the ratio of infeasibility and additional penalty costs compared with the nominal solution without demand uncertainty i.e., $\Gamma_i=0$ for I3 and I4, respectively. For each possible pair of $\sigma_i$ and $\Gamma_i$, we obtained a robust solution and evaluated the corresponding ratio of infeasibility and additional penalty cost. Then, we plot the corresponding points in Figure \ref{fig05} and Figure \ref{fig06}. They demonstrate that the ratio of infeasibility and the penalty costs are approximately inversely related.

\begin{figure}
\begin{subfigure}{.5\textwidth}
  \centering
\begin{tikzpicture}[only marks, y=1cm, scale=0.8]
\begin{axis}[   width=6cm,height=5cm,scale only axis,
                xlabel={penalty cost (\%)}, xmin=0, xmax=15, axis x line*=bottom, 
                xtick={0,5,10,15},
                xticklabels={0,5,10,15},
                xmajorgrids,
                ymin=0, ymax=30, ylabel={additional capacity (\%)},
                ytick={0,5,10,15,20,25,30},
                yticklabels={0,5,10,15,20,25,30},
                ymajorgrids,
      axis y line*=left,
       line width=1.0pt, mark size=2.0pt,
                legend style={at={(0.98,0.02)},anchor=south east,draw=black,fill=white,align=left},
                legend pos=north west
                ]
                                
\addlegendentry{$y=1.293x$, $R^2=0.793$};
\addplot [color=black, mark=*]
                coordinates  {(0,0) (0.16,0.0) (0.16,0.0) (0.42,0.0) (3.45,0.0) (4.70,2.5) (0.16,0.0) (4.73,2.5) (5.13,2.5) (5.62,4.1) (6.32,6.9) (0.50,0.0) (5.27,2.5) (6.36,6.9) (7.54,10.1) (8.51,10.3) (4.86,2.5) (6.36,6.9) (7.62,10.1) (8.72,13.3) (9.70,16.6) (5.14,2.5) (6.95,6.9) (8.77,13.3) (9.85,16.6) (13.17,23.0) 
};
\draw [-,semithick,black] (axis cs:0,0) -- (axis cs:20,25.86);  
\end{axis}              
\end{tikzpicture}    
  \caption{I3}
  \label{fig07}
\end{subfigure}
\begin{subfigure}{.5\textwidth}
  \centering
\begin{tikzpicture}[only marks, y=1cm, scale=0.8]
\begin{axis}[   width=6cm,height=5cm,scale only axis,
                xmin=0, xmax=15, xlabel={penalty cost (\%)},
                xtick={0,5,10,15},
                xticklabels={0,5,10,15},
                xmajorgrids,
                ymin=0, ymax=30, ylabel={additional capacity (\%)},
                ytick={0,5,10,15,20,25,30},
                yticklabels={0,5,10,15,20,25,30},
                ymajorgrids,
                axis lines*=left, line width=1.0pt, mark size=2.0pt,
                legend style={at={(0.98,0.02)},anchor=south east,draw=black,fill=white,align=left},
                legend pos=north west
                ]
\addplot [color=black, mark=*]
                coordinates  {(0,0) (0.2,0.0) (1.0,0.0) (1.6,0.0) (1.9,0.0) (2.7,0.0) (1.6,0.0) (2.7,3.0) (3.5,3.0) (4.1,9.3) (5.7,9.3) (1.8,3.0) (3.7,9.3) (5.0,6.3) (6.6,9.3) (8.0,9.3) (2.9,6.3) (5.0,6.3) (7.1,15.0) (9.1,18.1) (10.4,18.1) (3.5,3.0) (6.4,9.3) (9.2,18.1) (11.1,24.4) (12.2,24.4) 

};              
\draw [-,semithick,black] (axis cs:0,0) -- (axis cs:20,35.46);                  
\addlegendentry{$y=1.773x$, $R^2=0.875$};
\end{axis}              
\end{tikzpicture}    
  \caption{I4}
  \label{fig08}
\end{subfigure}
\caption{Relationship between the additional total capacities and the penalty cost, $\sigma_{j} \times \Gamma_{i}$ $\in$ $\{10\%,$ $20\%,$ $30\%,$ $40\%,$ $50\%\}$ $\times$ $\{1,$ $2,$ $3,$ $4,$ $5\}$.}
\label{fig0708}
\end{figure}
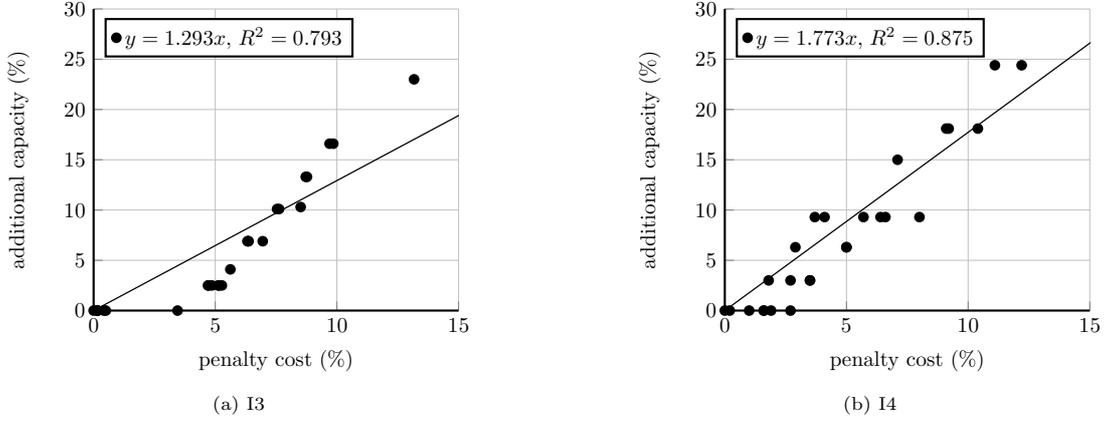

Figure \ref{fig07} and Figure \ref{fig08} illustrate relationship between additional total capacities and the penalty costs in comparison with the nominal solution without demand uncertainty, i.e., $\Gamma_i=0$ for I3 and I4, respectively. A robust solution may need to open additional facilities compared to the nominal solution to cope with uncertain demands. Such additional capacity and additional penalty cost are obtained for each possible pair of $\sigma_i$ and $\Gamma_i$ values, and they are plotted in Figure \ref{fig07} and Figure \ref{fig08}.  From the linear regression with setting the y-intercept at zero, we can see that the rate of additional total capacities is linearly correlated to the rate of penalty costs. The coefficient of determination $R^2$ is equal to 0.793 and 0.875 for I3 and I4, respectively. From this, we can confirm that the additional costs for robust solutions are directly related to the additional capacities.

\section{Conclusion} \label{section6} 

In this paper, we proposed a branch-and-price algorithm for the robust SSCFLP with the cardinality-constrained demand uncertainty set. The algorithm is based on the allocation-based mathematical model induced by the Dantzig-Wolfe decomposition. The pricing subproblem is the robust binary knapsack problem, which can be solved by solving nominal binary knapsack problems at most $n$ times. The computational results show that our proposed algorithm can solve practical instances better than CPLEX, which solves the MIP reformulation of the robust SSCFLP. We also verify that the trade-off between the robustness of the solutions and additional costs empirically by Monte-Carlo simulation studies. 

Further works may be required to improve the branch-and-price algorithm for the robust SSCFLP, and we suggest some of them. Efficient heuristics for the robust SSCFLP will be helpful as primal heuristics for the branch-and-price algorithm. Additionally, an efficient column management technique may help to reduce the size of the restricted master problem. Lastly, it may be worthwhile to adopt some techniques to improve the convergence speed, like the stabilized column generation. Moreover, considering other uncertainty sets of demands, e.g. polyhedral uncertainty set, ellipsoidal uncertainty set, can be interesting subjects for the robust SSCFLP.

\section*{Acknowledgments}

This work was supported by the National Research Foundation of Korea (NRF) Grant funded by the Korea government (MSIT) (No. 2019R1F1A1061361)

\end{document}